\documentclass{amsart}

\usepackage[utf8]{inputenc} 

\usepackage[T1]{fontenc}    % use new font-encoding (to have bold
\usepackage{ae,aecompl}     % changes to Type 1 fonts
\usepackage{amsmath}
\usepackage{amsfonts}
\usepackage{amssymb}
\usepackage{amsthm} % 
\usepackage{mathrsfs} % 
\usepackage{enumerate} 
\usepackage{color}
\usepackage{tikz}
\usepackage{pgfplots}
\usetikzlibrary{calc}
\usetikzlibrary{shapes}
\usepackage{tkz-base} 
\usepackage{tkz-euclide}
\usetkzobj{all}

\usepackage{graphicx}

\usepackage{caption}
\usepackage[hyperfootnotes]{hyperref}

\theoremstyle{plain} % style plain
\newtheorem{thm}{Theorem}[section]
\newtheorem{cor}[thm]{Corollary}
\newtheorem{prop}[thm]{Proposition}
\newtheorem{lem}[thm]{Lemma}

\theoremstyle{definition}
\newtheorem{defi}[thm]{Definition}

\newtheorem{remark}[thm]{Remark}

\newcommand{\cone}{\operatorname{cone}}
\newcommand{\conv}{\operatorname{conv}}
\newcommand{\Span}{\operatorname{span}}
\newcommand\OO{{\operatorname{O}}}

\author[C. Hohlweg]{Christophe~Hohlweg}
\address[Christophe Hohlweg]{Universit\'e du Qu\'ebec \`a Montr\'eal\\
LaCIM et D\'epartement de Math\'ematiques\\ CP 8888 Succ. Centre-Ville\\
Montr\'eal, Qu\'ebec, H3C 3P8\\ Canada}
\email{hohlweg.christophe@uqam.ca}
\urladdr{http://hohlweg.math.uqam.ca}

\author[J.-P.~Pr\'eaux]{Jean-Philippe Pr\'eaux}
\address[Jean-Philippe Pr\'eaux]{Laboratoire d'Analyse, Topologie et Probabilit\'es, UMR CNRS 7373\\ 39 rue Joliot-Curie \\13453 Marseille cedex, France}
\email{jean.philippe.preaux@gmail.com}
\urladdr{http://www.i2m.univ-amu.fr/perso/jean-philippe.preaux}

\author[V. Ripoll]{Vivien~Ripoll}
\address[Vivien Ripoll]{Fakult\"at f\"ur Mathematik, Universit\"at Wien\\ Oskar-Morgenstern-Platz 1\\ 1090 Wien\\Austria}
\email{vivien.ripoll@univie.ac.at}
\urladdr{http://www.normalesup.org/~vripoll}

\title[Limit Set of Root Systems of Coxeter Groups on Lorentzian spaces]{On the Limit Set of Root Systems\\ of Coxeter Groups acting on Lorentzian spaces}

\begin{document}

\begin{abstract}
The notion of limit roots of a Coxeter group $W$ was  recently introduced: they are the accumulation points of directions of roots of a root system for $W$. In the case where the root system lives in a Lorentzian space, i.e.,  $W$ admits a faithful representation as a discrete reflection group of isometries of a hyperbolic space, the accumulation set of any of its orbits is then classically called the limit set of $W$. In this article, we show that the set of limit roots of a Coxeter group~$W$ acting on a Lorentzian space is equal to the limit set of~$W$ seen as a discrete reflection group of hyperbolic isometries. 
\end{abstract}

\keywords{Root system, Coxeter group, limit roots, limit set of discrete group, Lorentzian space, discrete reflection group, Apollonian gasket, hyperbolic Coxeter group, hyperbolic isometries, Kleinian groups}

\date{\today}

%%%%%%%%%%%%%%%%%%%%%%%%%%%%%%%%%

 \maketitle

%\tableofcontents

%%%%%%%%%%%%%%%%%%%%%%%%%%%%%%%%%%%%%%%%%%%%%%%%%%%%%%%%
\section{Introduction}\label{se:Intro}
%%%%%%%%%%%%%%%%%%%%%%%%%%%%%%%%%%%%%%%%%%%%%%%%%%%%%%%%

Coxeter groups acting on a Euclidean vector space as a linear reflection group are precisely finite reflection groups~\cite{Bo68,Hu90}. In this case, the relation between reflection hyperplanes and the set of their normal vectors called {\em root system} is well understood and their interplay is the main tool to study those groups. Surprisingly, the duality between roots and reflection hyperplanes is not very well exploited to study other cases of Coxeter groups.  For instance, it is the case for the class of Coxeter groups seen as discrete groups generated by reflections of hyperbolic spaces. 

In order to fill this gap,  the language of {\em limit roots} and {\em imaginary cone} was initiated in~\cite{HoLaRi12,Dy12,DyHoRi13}. The main goal of this article is to translate this language of limit roots and imaginary cones into the language of hyperbolic geometry. As a byproduct, we show that the set of limit roots  is equal to the limit set of the corresponding Coxeter group seen as a discrete reflection group of hyperbolic isometries. Let us describe more precisely the content of this article. 

On the one hand, any Coxeter group $W$ has a representation as a {\em discrete reflection subgroup} of the orthogonal group $\OO_B(V)$, where $V$ is a finite dimensional vector space and $B$ is a symmetric bilinear form. With such a representation of $W$ arises a natural set of vectors $\Phi$ called a {\em root system}, which are unit $B$-normal vectors of the reflection hyperplanes associated to reflections in $W$. The root system $\Phi$ has an empty set of accumulation points, but the projective version $\widehat\Phi$ of $\Phi$, represented on an affine hyperplane $\mathcal H$, has an interesting set of accumulation points $E(\Phi)$ (see for instance Figure~\ref{fig:intro}). The set $E(\Phi)$ is called the {\em  set of limit roots of $\Phi$} and its study was initiated in~\cite{HoLaRi12} and continued in~\cite{Dy12,DyHoRi13}. Among other properties, it was shown that $E(\Phi)$ lies on the  isotropic cone~${{Q=\{v\in V\,|\, B(v,v)=0\}}}$, that $E(\Phi)$ satisfies some fractal-like properties, that the convex hull $\conv(E(\Phi))$ of $E(\Phi)$ in an appropriate chart admits an infinite tiling called the {\em imaginary convex set}, and the action of $W$ on $\widehat\Phi\sqcup\conv(E(\Phi))$ was studied. 
\begin{figure}[!h]
\begin{minipage}[b]{\linewidth}
\centering
\scalebox{1}{\input{FigIntro1.tex}}
\end{minipage}%
\caption{Picture of the first normalized roots of the root system $\Phi$ with diagram in the upper left of the picture and  converging toward the set of limit roots $E(\Phi)$. Under the action of $W$ on $E(\Phi)$, the simple reflection $s_\alpha$ (for a simple root $\alpha$) maps $x\in E(\Phi)$ to $s_\alpha\cdot x$ where $L(\alpha,x)\cap Q=\{x,s_\alpha\cdot x\}$; note that if the line $L(\alpha,x)$ is tangent to $Q$, then $s_\alpha\cdot x = x$.  }
\label{fig:intro}
\end{figure}

On the other hand, when $B$ has signature $(n,1)$, the pair $(V,B)$ is called a {\em Lorentzian $(n+1)$-space} and $Q$ is called the {\em light cone}, see for instance~\cite[Chapter~3]{Ra06}. In this case, $W$ is a discrete reflection subgroup of the group $\mathcal I(\mathbb H^{n})$ of isometries of the hyperbolic $n$-space  $\mathbb H^{n}$ (see~\cite{Ra06}): there is an affine hyperplane~$\mathcal H$ of~$V$ such that $Q\cap \mathcal H$ is a sphere, and the corresponding open ball ($Q^-\cap \mathcal H$) is naturally identified with the {\em projective model}  $\mathbb H^{n}_p$  for $\mathbb H^{n}$  (see details in~\S\ref{se:Hyperbo}). We will explain in \S\ref{se:Coxeter} the relation between the set $E(\Phi)$ of limit roots and the imaginary convex set $Z$ (the projective version of Dyer's imaginary cone from \cite{Dy12}): the closure $\overline{Z}$ is equal to the convex hull of $E(\Phi)$, and the set $E(\Phi)$ is equal to the intersection of $\overline{Z}$ with the boundary of $\mathbb H^{n}_p$, see Figure~\ref{fig:3}.
%% We will explain in \S3 that the imaginary convex set is a tiling of a part of $\mathbb H^{n}_p$ on which $W$ acts faithfully and that the set $E(\Phi)$ of limit roots is nothing else than the intersection of the closure of this tiling with the boundary of $\mathbb H^{n}_p$, see Figure~\ref{fig:3}.

In the context of hyperbolic geometry, another notion of limit is the {\em limit set of $W$}, denoted by $\Lambda(W)$, which is defined to be the accumulation set of the $W$-orbit of a point $x\in \mathbb H^{n}_p$. The limit set $\Lambda(W)$ is of great interest in the theory of Kleinian groups and its  generalizations, see for instance~\cite{Nic89,MaTa98,Ra06}.  In~\cite{HoLaRi12,DyHoRi13} some examples of the limit roots associated to a Coxeter group, acting as a discrete reflection group on a Lorentzian space, looked like the limit set of a Kleinian group; see for instance the Apollonian circles in Figure~\ref{fig:intro} obtained as limit set of a root system. As a byproduct of our discussion we obtain the following theorem. 

\begin{thm}\label{thm:main} If $(V,B)$ is a Lorentzian space,  the limit set  $E(\Phi)$ of the root system~$\Phi$ is equal to the limit set $\Lambda(W)$.
\end{thm}

A proof of Theorem~\ref{thm:main} follows by reinterpreting the results in \cite{DyHoRi13} into the language of hyperbolic geometry, as explained in details in \S\ref{ss:prooftheorem}. The idea is as follows. Each hyperplane of a reflection in $\mathbb H ^n$ corresponds to a unique space-like vector --- a positive root --- in the ambient Lorentzian space. Normalizing these positive roots, we obtain $E(\Phi)$ as their set of accumulation points. The interior of the convex hull $\conv(E(\Phi))$ of $E(\Phi)$ is contained in $Q^-$. In \cite{DyHoRi13}, it is shown that $E(\Phi)$ is also the accumulation set of the $W$-orbit of any point in $\conv(E(\Phi)) \cap Q^-$. The proof then follows by showing that $\mathbb H^n \cap \conv(E(\Phi))$ is non empty and, therefore, that $E(\Phi)$ is the accumulation set of a $W$-orbit of a point in $\mathbb H^n$.

The proof relies heavily on the equality $\conv(E(\Phi))\cap Q=E(\Phi)$  proven in \cite{DyHoRi13} in the case of $B$ with signature $(n,1)$. The question of the validity of this equality  for $B$ of arbitrary signature and $W$ irreducible (i.e., the Coxeter graph is connected) is still open.

We do not know a direct proof of this statement using only tools from hyperbolic geometry. It is well known that for a point $x$ in the hyperbolic space (i.e., a time-like vector in $V$), the limit set of the orbit of $x$ does not depend on the choice of~$x$ and is (by definition) the set
$\Lambda(W)$. Theorem~\ref{thm:main} implies that this is also the case when replacing $x$ with any root of $\Phi$ (but a root is always a space-like
vector in $V$). These facts open the natural question about what are the limit sets of orbits of other space-like vectors. After a first version of this article appeared on the Arxiv, H.~Chen and J.-P.~Labb\'e showed in~\cite{ChLa14} that the accumulation set of the $W$-orbit of a
space-light vector is not in general contained in the isotropic cone $Q$, and therefore is different from $\Lambda(W)$ and $E(\Phi)$, see~\cite[Figure 2]{ChLa14}.  Even more intriguing is the fact that for a root system living in a quadratic space of arbitrary signature, there may exist \emph{isotropic} vectors that are in the accumulation set of some orbit but that are not limit roots, see \cite[Example 3.11]{ChLa14}.

The last section~\S\ref{ss:univers} of this article is devoted  to an explanation of the example in Figure~\ref{fig:intro} together with its relation with Apollonian gaskets. 

 We aim for this article to be accessible both to the community familiar with reflection groups and root systems and to the community familiar with discrete subgroups of isometries in hyperbolic geometry,  so we will make a point to properly survey the objects and constructions mentioned above. In particular, in \S\ref{see:Vinberg}, we discuss and make precise the different occurrences of the word ``hyperbolic'' in the context of Coxeter groups.

%%%%%%%%%%%%%%%%%%%%%%%%%%%%%%%%%%%%%%%%%%%%%%%%%%%%%%%%
\section{Lorentzian and Hyperbolic Spaces}\label{se:Hyperbo}
%%%%%%%%%%%%%%%%%%%%%%%%%%%%%%%%%%%%%%%%%%%%%%%%%%%%%%%%

The aim of this section is to survey the background we need on hyperbolic geometry. The presentation of the material  in this section is mostly based  on~\cite[Chapter 3 and \S6.1]{Ra06}, see also~\cite[Chapter A]{BP92}, \cite{Vi93g} and~\cite[Chapter 6]{Da08}. 

\smallskip

Let $V$ be a real vector space of dimension~$n+1$ equipped with a symmetric bilinear form~$B$.  We will denote by $q(\cdot)=B(\cdot, \cdot)$ the quadratic form associated to~$B$ and by $Q:=\{x\in V\,|\, q(x)=0\}$ the {\em isotropic cone of} $B$, or equivalently, of $q$.  

\subsection{Lorentzian  spaces}
Suppose from now on that the signature of $B$ is $(n,1)$.  
The pair $(V,B)$ is then called a {\em Lorentzian $(n+1)$-space} and $Q$ is called the {\em light cone}. Moreover, the elements in the set  $Q^-:=\{v\in V\,|\, q(v)<0\}$ are said to be {\em time-like}, while  the elements in $Q^+:=\{v\in V\,|\, q(v)>0\}$ are {\em space-like}\footnote{This vocabulary is inspired from the theory of relativity, where $n=3$.};  see the top picture in Figure~\ref{fig:2} for an illustration.

A {\em Lorentz transformation\footnote{These transformations are
    called {\em $B$-isometries} in~\cite{HoLaRi12,DyHoRi13}, since in
    these articles $B$ does not necessarily have signature~$(n,1)$.} }
is a map on $V$ that preserves $B$. So, in particular, a Lorentz
transformation preserves $Q$, $Q^+$ and $Q^-$. It turns out that
Lorentz transformations are linear isomorphisms on $V$ (indeed, one can
prove that a map preserving a nondegenerate bilinear form is
linear). We denote by $\OO_B(V)$ the set of Lorentz transformations of
$V$:
\[
\OO_B(V):=\{f\in\textrm{GL}(V)\,|\, B(f(u),f(v))=B(u,v),\ \forall u,v\in V\}.
\]
The well-known {\em Cartan-Dieudonn\'e Theorem} states that, since $B$ is non-de\-ge\-ne\-rate, an element of $\OO_B(V)$ is a product of at most {\em $(n+1)$ $B$-reflections}: for  a non-isotropic vector $\alpha \in V\setminus Q$, the \emph{$B$-reflection associated to~$\alpha$} (simply called reflection 
when~$B$ is unambiguous) is defined by the equation\footnote{Observe that if $B$ was positive definite, this equation would  be the usual formula for a Euclidean reflection.}
\begin{equation}
\label{eq:Reflection}
s_{\alpha} (v) = v - 2\frac{B(\alpha,v)}{B(\alpha,\alpha)} \ \alpha , \quad \text{for any   } v\in V. 
\end{equation}
We denote by $\alpha^\perp:=\{v\in V\,|\, B(\alpha,v)=0\}$ the orthogonal of the line $\mathbb R\alpha$ for the form $B$. Since 
${B(\alpha,\alpha)\neq 0}$, we have $\alpha^\perp\oplus \mathbb R \alpha = V$. It is straightforward to check that~$s_\alpha$ fixes 
$\alpha^\perp$ pointwise and that~${s_\alpha(\alpha)=-\alpha}$.

\subsection{Hyperbolic spaces}\label{ss:bases} We fix a basis $\mathcal B=(e_1,\dots,e_{n+1})$ of $V$ such that $q(v)=x_1^2+x_2^2+\dots x_n^2-x^2_{n+1}$, for any $v\in V$ with coordinates $(x_1,\dots,x_{n+1})$ in the basis $\mathcal B$. Equipped with this basis, $V$ is often denoted by $\mathbb{R}^{n,1}$.   The quadratic hypersurface $\{v\in V\,|\, q(v)=-1\}$, called the \emph{hyperboloid}, consists of time-like vectors and has two sheets. It is interesting to note that it is a differentiable surface (as the preimage of a regular value by the differentiable map $q:V\longrightarrow \mathbb{R}$) and is naturally endowed with a Riemannian metric because $B(.,.)$ restricted to the tangent spaces of each sheet is positive definite.
A {\em time-like vector $v$ is positive} if $x_{n+1}>0$; the positive sheet,
\[
\mathbb H^{n}:= \{v\in V\,|\, q(v)=-1\textrm{ and }x_{n+1}>0 \}
 \]
turns out to be a simply connected complete Riemannian manifold with constant sectional curvature equal to $-1$ (cf. \cite[Theorem A.6.7]{BP92}). This is  {\em the hyperboloid model of the  hyperbolic $n$-space}, see Figure~\ref{fig:2}. The distance function $d$ on $\mathbb H^n$ satisfies the equation $\cosh d(x,y)=-B(x,y)$.

\subsubsection{Group of isometries}\label{sss:Isom}
Observe that the group $\OO_B(V)$ acts on the quadratic hypersurface $\left\lbrace v\in V\,|\, q(v)=-1\right\rbrace$. A Lorentz transformation is {\em a positive Lorentz transformation} if it maps time-like positive vectors to time-like positive vectors. So the group $O_B^+(V)$ of positive Lorentz transformations preserves $\mathbb{H}^n$ and its distance, and the group of isometries $\mathcal I(\mathbb H^n)$ of $\mathbb H^n$ is isomorphic to $O_B^+(V)$: any isometry of $\mathbb H^n$ is the restriction to $\mathbb H^n$ of a positive Lorentz transformation. Moreover, it is well known that $\mathcal I(\mathbb H^n)$ is  generated by {\em hyperbolic reflections through hyperbolic hyperplanes}, of which we now recall the definition.

\subsubsection{Hyperbolic reflections}\label{sss:Ref} A linear subspace $F$ of $V$ is said to be {\em time-like} if $F\cap Q^-\not =\varnothing$, otherwise it is called {\em space-like}, or {\em light-like} if it contains an isotropic vector. A {\em hyperbolic hyperplane} is the intersection of $\mathbb H^n$ with a time-like hyperplane of $V$.  Let $H$ be a linear hyperplane in $V$ and $\alpha\in V$ be a normal vector to $H$ for the form $B$. Suppose $\alpha\notin Q$, so that $H\oplus \mathbb R \alpha = V$. Then, since $B$ has signature~$(n,1)$, we obtain that~$H$ is time-like if and only if  $\alpha\in V$ is  a space-like vector.  A reflection $s_\alpha\in \OO_B(V)$ is a {\em hyperbolic reflection} if  $\alpha^\perp=H$ is a time-like hyperplane or, equivalently, if $\alpha$ is a space-like vector of $V$. In this case, $s_\alpha\in\OO^+_B(V)$ and it restricts to an isometry of~$\mathbb H^n$.    

\begin{remark}
The fact that $s_\alpha\in \OO_B^+(V)$ for a space-like vector $\alpha$ follows from the fact that a reflection $s_\alpha$ is continuous and that it exchanges the two sheets (i.e. connected components) of the quadratic surface $\{v\in V\,|\, q(v)=-1\}$ if and only if~$\alpha^\perp$ is space-like, i.e.,  if and only if $\alpha$ is a time-like vector.
\end{remark}

%%%%%%%%%%%%%%%%%%%%%%%%%%%%%%%%%%%%%%

\subsection{The projective model}
\label{ss:projmodel}
 To make clear the link between hyperbolic geometry and  the results of \cite{HoLaRi12,DyHoRi13}, we need to introduce another model for $\mathbb H^n$. Consider the unit open (Euclidean) $n$-ball embedded in the affine hyperplane~${{\mathbb{R}^n\times\{1\}}}$ of $V$:
 \[
 D_1^n=\{v\in V\,|\, x_{n+1}=1\textrm{ and } x_1^2 + \dots +x_n^2<1\}
 \]
 and the map
 $p$ from $D_1^n$ to $\mathbb H^n$, called the {\em  radial projection} 
 \[
p: D_1^n\to \mathbb H^n
\]
 where $p(v)$ is the intersection point of the line $\mathbb Rv$  with $\mathbb H^n$ (see Figure~\ref{fig:2}).  A simple calculation shows that
 \[
 p(v)=\frac{v}{\sqrt{|q(v)|}}.
 \]
The unit ball $D_1^n$ endowed with the pullback metric with respect to $p$, i.e. which makes $p$ an isometry, 
 is a (non conformal) model $\mathbb{H}^n_p$ for $\mathbb H^n$ called the {\em projective ball model}\footnote{This model is also sometimes called the {\em Beltrami-Klein model} in the literature.}, see~\cite[\S6.1]{Ra06}. 
  
   First, observe that using the equation for $q$ in the basis~$\mathcal B$, we have that $D_1^n\subseteq Q^-$.  Let $\mathcal H$ be the affine hyperplane directed by $\Span(e_1,\dots,e_n)$ and passing through the point $e_{n+1}$, then  we get 
 \[
 D_1^n = Q^-\cap \mathcal H,
 \] 
with boundary  $Q\cap \mathcal H$.  The next proposition follows from the previous discussion and \cite[Equation (6.1.2)]{Ra06}.

\begin{prop}
 \label{prop:proj}
  The {\em projective model $\mathbb{H}^n_p$} has underlying space $D_1^n=Q^-\cap \mathcal H$  and its boundary $\partial \mathbb{H}^n_p$ is $Q\cap \mathcal H$. Moreover, $p:\mathbb{H}^n_p \to \mathbb H^n$ is an isometry whose inverse is
  \[
  p^{-1}(v)=\frac{v}{x_{n+1}}=(x_1/x_{n+1},\dots,x_n/x_{n+1},1). 
  \]
  \end{prop}
This proposition is illustrated for $n+1=2$ and $n+1=3$ in Figure~\ref{fig:2}.

\begin{figure}[!h]
\begin{minipage}[b]{\linewidth}
\centering
\includegraphics[scale=0.6]{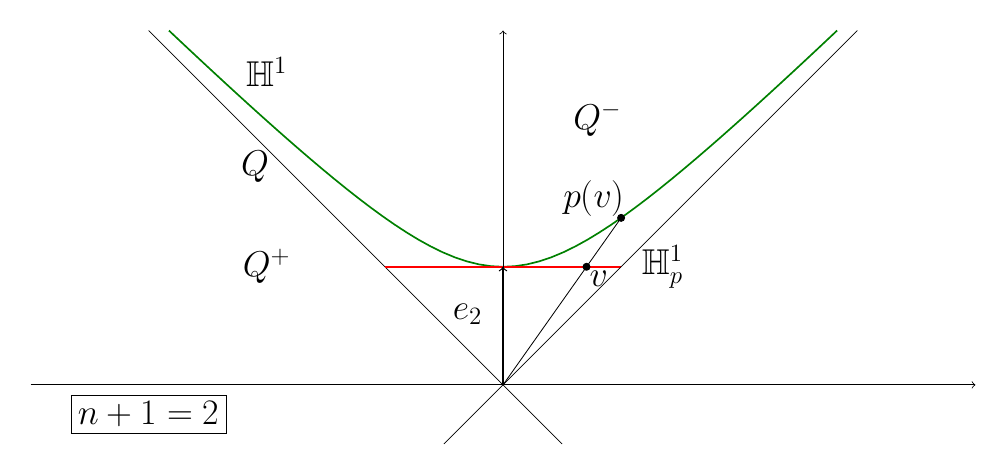}
\includegraphics[scale=1]{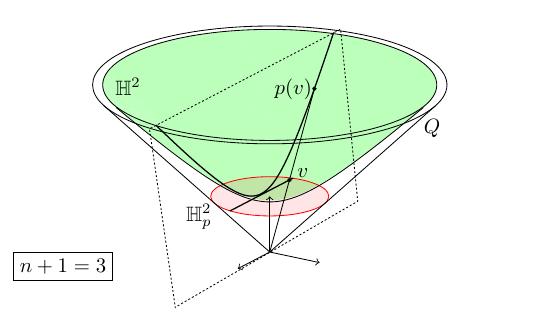}

\end{minipage}%
\caption{Pictures of Lorentzian spaces of dimension $n+1=2$ and $n+1=3$ with the hyperbolic spaces $\mathbb H^n$, the projective model $\mathbb{H}^n_p$, the radial projection $p$ and a time-like plane $H$ cutting $\mathbb{H}^2$ and $\mathbb{H}^2_p$.}
\label{fig:2}
\end{figure}

\subsubsection{Hyperplanes, reflections and isometries} The projective model gives us an easy description of hyperplanes: a {\em hyperbolic hyperplane in $\mathbb{H}_p^n$} is simply the intersection of a time-like linear hyperplane of $V$ with $\mathbb{H}_p^n$. Let $\mathcal I(\mathbb{H}_p^n)$ be the group of isometries of $\mathbb{H}_p^n$. 

\begin{cor}\label{cor:action}
\label{prop:projisom} 
The conjugation by $p$ is an isomorphism from $\mathcal I(\mathbb{H}^n)$ to $\mathcal I(\mathbb H_p^n)$: for $\varphi\in \mathcal I(\mathbb{H}^n)$ and a point $v\in \mathbb{H}_p^n$,  $\varphi\cdot v:=p^{-1}\circ \varphi\circ p(v) $ defines the isometric action of $\varphi$ on~$\mathbb{H}_p^n$.  Moreover   
$
\varphi\cdot v 
$
 is the intersection point of the linear line $\mathbb R \varphi(v)$  with the ball $D_1^n$. 
 
 In particular, if  $\alpha\in V$ is a space-like vector, then  $s_{\alpha}\cdot v=p^{-1} \circ s_\alpha  \circ p (v)$ is the hyperbolic reflection in $\mathcal I(\mathbb{H}_p^n)$ of $v$ through the time-like hyperplane $\alpha^\perp$. 
\end{cor}
\begin{proof}  The property of the conjugation by $p$ follows immediately from Proposition~\ref{prop:proj}. For the characterization of $\varphi \cdot v$ as the intersection of $\mathbb R \varphi(v)$ and $D_1^n$, note that $\varphi$ acts on $V$ linearly, and $p$ preserve directions. So $\varphi \cdot v$ is colinear with~$\varphi(v)$. Moreover by definition $\varphi \cdot v$ lies also in $\mathbb H_p^n$, i.e., in $D_1^n$.
\end{proof}

\subsection{Limit sets of discrete groups of hyperbolic isometries}\label{ss:LimitSet}

The notion of \emph{limit set} is central to study the
dynamics of discrete groups of hyperbolic isometries: it provides an
interesting topological space on which the group naturally acts, and
the properties of the limit set characterize some properties of the
group. We recall below the definition and some basic features of limit
sets; see for example \cite[\S12.1]{Ra06}, \cite[\S2]{Vi93} or \cite[\S1.4]{Nic89} for
details.

Given a hyperbolic isometry $\varphi$ of $\mathbb{H}_p^n$ (with underlying space $D_1^n$),
 we extend the action of $\varphi$ to the closed ball
$\overline{D_1^n}$. Note that $\varphi$ preserves the
boundary $\partial \mathbb{H}^n_p=Q\cap \mathcal H$ of $\mathbb{H}_p^n$.  

 Let $\Gamma\subseteq \mathcal I(\mathbb{H}_p^n)$ be a discrete group of hyperbolic isometries. For $x$ a point in the closed ball $\overline{D_1^n}$, the following are equivalent:
\begin{itemize}
\item $x$ is an accumulation point of the orbit $\Gamma\cdot x_0$ for some $x_0\in \mathbb{H}_p^n$;
\item $x$ is an accumulation point of the orbit $\Gamma\cdot x_0$ for any $x_0\in \mathbb{H}_p^n$;
\item $x$ is in $\partial \mathbb{H}^n_p\cap \overline{\Gamma\cdot x_0}$ for some $x_0\in \mathbb{H}_p^n$.
\end{itemize}
Such a point is called a \emph{limit point} of $\Gamma$.

\begin{defi}  The set $\Lambda(\Gamma)$ of limit points is called \emph{the limit set of $\Gamma$}.
\end{defi}

The fact that an orbit has no accumulation points inside the open ball $D_1^n$ is clear, since the group is discrete and $\mathcal{I}(\mathbb{H}_p^n)$ acts properly on $\mathbb{H}_p^n$. The fact that the limit set of an orbit does not depend on the chosen point follows from the relation between hyperbolic and Euclidean distances, see \cite[Theorem~12.1.2]{Ra06}.

\smallskip

The limit set $\Lambda(\Gamma)$ is clearly closed and $\Gamma$-stable. Many general properties are
known for limit sets of discrete groups of hyperbolic isometries. For example, either $\Lambda(\Gamma)$ is finite, in which
case $|\Lambda(\Gamma)|\leq 2$% and $\Gamma$ has a finite orbit in the closed ball $\overline{D_1^n}$
, see \cite[Theorem~12.2.1]{Ra06}, or 
$\Lambda(\Gamma)$ is uncountable and the action of $\Gamma$ on $\Lambda(\Gamma)$ is minimal. In \cite{DyHoRi13}, the authors show an analogous property  for the set of limit roots of a root system, which we will define below.

%%%%%%%%%%%%%%%%%%%%%%%%%%%%%%%%%%%%%%%%%%%%%%%%%%%%%%%%
\section{Coxeter Groups and Hyperbolic Geometry }\label{se:Coxeter}
%%%%%%%%%%%%%%%%%%%%%%%%%%%%%%%%%%%%%%%%%%%%%%%%%%%%%%%%

The aim of this section is to translate constructions and results related to root systems, and that are taken out from~\cite{Dy12,HoLaRi12,DyHoRi13}, into the language of hyperbolic geometry. This naturally leads  to a proof of Theorem~\ref{thm:main}.

\smallskip
 
 Recall that a Coxeter system $(W,S)$ is such that $S\subseteq W$ is a set of generators for the Coxeter group $W$, subject only to relations of the form $(st)^{m_{s,t}}=1$, where $m_{s,t}\in \mathbb N^* \cup \{\infty\}$ is attached to each pair of generators $s,t\in S$, with 
$m_{s,s}=1$ and $m_{s,t}\geq 2$ for $s\neq t$. We write $m_{s,t}=\infty$ if the product $st$ is of infinite order. In the following, 
we always suppose that the set of generators $S$ is finite. If all the $m_{s,t}$, for $s\not = t$  are $\infty$ then we say that $W$ is a {\em universal Coxeter group}.

It turns out that any Coxeter group can be  represented as a {\em discrete reflection subgroup} of $\OO_B(V)$ for a certain pair $(V,B)$: Coxeter groups are  the discrete reflection groups associated to {\em based root systems}~\cite[Theorem 2]{Vi71} (see~\cite[\S1 and Theorem 1.2.2]{Kr09} for a recent exposition of this result).
\medskip

%% From now on, we fix $(V,B)$ to be a Lorentzian $(n+1)$-space.  Note that most of the results from~\cite{Dy12,HoLaRi12,DyHoRi13} we recall in the two next sections are  valid  for arbitrary $(V,B)$. 

\subsection{Based root systems and geometric representations}\label{subsection:basedrootsystems}
 Let us cover the basics on based root systems (see for instance~\cite[\S1]{HoLaRi12} for more details). A {\em simple system}  $\Delta$ is a finite subset of  $V$ such that:
\begin{enumerate}[(i)]
\item $\Delta$ is positively independent: if $\sum_{\alpha \in
    \Delta} \lambda_{\alpha} \alpha =0$ with all $\lambda_\alpha \geq
  0$, then all $\lambda_\alpha=0$;
\item for all $\alpha, \beta \in \Delta$, with $\alpha \neq \beta$,
  $\displaystyle{B(\alpha,\beta) \in \ (-\infty,-1] \cup
    \{-\cos\left(\frac{\pi}{k}\right), k\in \mathbb Z_{\geq 2} \} }$;
\item for all $\alpha \in \Delta$, $B(\alpha,\alpha)=1$.
\end{enumerate}
 Denote by $S:=\{s_\alpha \,|\, \alpha \in \Delta\}$ the set of~$B$-reflections 
associated to elements in $\Delta$. Let $W$ be the subgroup of $\OO_B(V)$ generated by 
$S$, and $\Phi=W(\Delta)$ be the orbit of~$\Delta$ under the action of~$W$. 

The pair $(\Phi,\Delta)$ is called a \emph{based root system in $(V,B)$}. For simplification we will often use simply the term {\em root system}. Its set of {\em positive roots} is $\Phi^+:=\cone(\Delta)\cap \Phi$, and like for classical root systems, we have the property $\Phi=\Phi^+ \sqcup (-\Phi^+)$. The {\em rank} of $\Phi$ is the cardinality of~$\Delta$, i.e., the cardinality  of $S$.   Vinberg~\cite[Theorem 2]{Vi71} shows that $(W,S)$ is always a Coxeter system and $W$ is a discrete reflection group in~$\OO_B(V)$. Such a representation of a Coxeter group is called {\em a geometric representation}. Conversely, it is well known that any (finitely generated) Coxeter group can be geometrically represented with a root system, see for instance \cite[Chapter 5]{Hu90}.

\subsection{Representations as discrete reflection groups of hyperbolic isometries}
In this subsection, we restrict to the case where $(V,B)$ is a Lorentzian $(n+1)$-space. A geometric representation of a Coxeter group $W$ as a discrete subgroup of $\OO_B(V)$ then yields a faithful representation of~$W$ as a discrete subgroup of isometries of $\mathbb{H}^n$ that is generated by hyperbolic reflections. Therefore, by conjugation by the radial projection (Corollary~\ref{prop:projisom}), it also provides a faithful representation of~$W$ as a discrete subgroup of $\mathcal I(\mathbb{H}_p^n)$ generated by reflections.

The key is to observe that $W\subseteq \OO_B(V)$ is in fact a subgroup of $\OO^+_B(V)$, the group of positive Lorentz transformations: from \S\ref{sss:Isom}, we know that $\OO^+_B(V)$ is isomorphic to $\mathcal I(\mathbb{H}^n)$  by restriction to $\mathbb{H}^n$. 

\begin{prop}\label{prop:repHypo} Let $(\Phi,\Delta)$ be a based root system in the Lorentzian $(n+1)$-space $(V,B)$ with associated Coxeter system $(W,S)$.  Then $W\subseteq \OO^+_B(V)$ and this geometric action of $W$ on the $(n+1)$-Lorentzian space preserves $\mathbb{H}^n$. This yields a restricted representation of $W$  on $\mathcal{I}(\mathbb{H}^n)$ that is faithful and discrete. Consequently,  the projective action\footnote{From Corollary~\ref{cor:action}.} of $W$ on $\mathbb{H}_p^n$ is also faithful and discrete. 

Moreover, the action of $W$ on $\mathbb{H}^n$ (resp. $\mathbb{H}^n_p$) is  generated by reflections through the hyperbolic hyperplanes $\alpha^\perp\cap \mathbb{H}^n$ (resp.  $\alpha^\perp\cap \mathbb{H}^n_p$) for all $\alpha\in \Delta$.
\end{prop}
\begin{proof}
Since $\Delta$ is constituted of space-like vectors, the hyperplanes $\alpha^\perp$ are time-like. From~\S\ref{sss:Ref} we know therefore that $s_\alpha\in \OO^+_B(V)$ for all  $\alpha\in\Delta$. Since $W$ is generated by $S=\{s_\alpha\,|\, \alpha\in\Delta\}$, we have necessarily that $W\subseteq \OO^+_B(V)$.
\end{proof}

From now on, we denote  $\mathrm{H}_\alpha= \alpha^\perp\cap\mathbb{H}^n$ for any space-like vector.

\begin{remark} \label{prop:fundamentallink} The relative position between two hyperbolic hyperplanes has a nice characterization using their associated roots, which can be viewed as their ``normal vectors''. Let $\alpha$ and $\beta$ be two space-like linearly independent vectors such that $B(\alpha,\alpha)=B(\beta,\beta)=1$. So we know that $\alpha^\perp$ and $\beta^\perp$ are time-like hyperplanes. With the notations $\mathrm{H}_\alpha=
\alpha^\perp\cap\mathbb{H}^n$ and $\mathrm{H}_{\beta}=\beta^\perp\cap \mathbb{H}^n$ we have:
\begin{itemize}
\item[(i)] $\mathrm{H}_\alpha$ and $\mathrm{H}_{\beta}$ {\em intersect} if and only if $B(\alpha,\beta)\in ]-1,1[$, and in this case their dihedral angle is $\arccos |B(\alpha,\beta)|$;
\item[(ii)] $\mathrm{H}_\alpha$ and $\mathrm{H}_{\beta}$ are {\em parallel} if and only if $B(\alpha,\beta)=\pm 1$;
\item[(iii)] $\mathrm{H}_\alpha$ and $\mathrm{H}_{\beta}$ are {\em ultra-parallel} if and only if $|B(\alpha,\beta)|>1$ and in such a case  their distance in $\mathbb{H}^n$ is $\mathrm{arccosh}\,(|B(\alpha,\beta)|)$.  
\end{itemize}
Statement (i) follows from Theorem 3.2.6 (see also the discussion that follows in \S 3.2 and \S 6.4) of \cite{Ra06}; statement  (ii) follows from Theorem 3.2.9 of \cite{Ra06}, and statement (iii) from Theorems 3.2.7 and 3.2.8 of \cite{Ra06}.)
\end{remark}

\subsection{Limits of roots}\label{ss:limitroots}

 The constructions and statements of this subsection and the following one aim to be applied to the case of a based root system $(\Phi, \Delta)$ in a Lorentzian $(n+1)$-space  $(V,B)$ corresponding to a discrete reflection group of hyperbolic isometries.  Such based root systems are known as {\em weakly hyperbolic based root systems}; this connection will be made explicit in \S\ref{see:Vinberg}.
 
 \begin{defi}
A based root system $(\Phi,\Delta)$ in  a quadratic space $(V,B)$ is  {\em weakly hyperbolic}  if  $\Span(\Delta)$, together with the restriction of $B$ to $\Span(\Delta)$, is a Lorentzian space.
\end{defi}
 
Let $(\Phi, \Delta)$ be a based root system in a Lorentzian $(n+1)$-space  $(V,B)$. To simplify the arguments and definitions we always assume from now on that $\Span(\Delta)=V$ and that~$W$ is an irreducible Coxeter group. Therefore $(\Phi, \Delta)$ is infinite and is a based root system. For more details, see~\cite{Dy12,HoLaRi12,DyHoRi13}.

 For an arbitrary norm on the vector space $V$, the norm of any injective sequence of roots goes to infinity, so $\Phi$ does not have accumulation points (see~\cite[Theorem 2.7]{HoLaRi12}). We rather look at accumulation points of the directions of the roots. In order to do this, we will cut those directions by {\em a hyperplane transverse to $\Phi^+$}, i.e. an affine hyperplane that intersects all the directions of the roots; see \cite[\S2.1]{DyHoRi13} for the existence of this affine hyperplane. So we obtain points that are representatives of those directions, see for instance Figure~\ref{fig:dihed}.
  
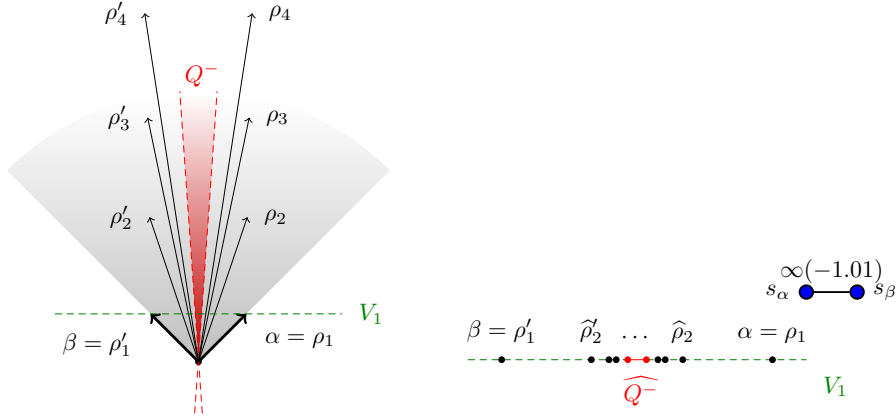
\begin{figure}[!ht]
\centering
\captionsetup{width=0.9\textwidth}   
\scalebox{0.9}{\begin{tabular}{c@{\hspace{1cm}}c}

\begin{tikzpicture}
	[scale=1,
	 pointille/.style={densely dashed},
	 axe/.style={color=black, very thick},
	 rotate=45]

\coordinate (O) at (0,0);
\fill (O) circle (0.05);

\def\arctandemi{40.97}

\shade [shading=axis,top color=white!75!black,bottom color=white,shading angle=180] (O) -- (4,0) arc (0:\arctandemi:4) -- (O);
\shade [shading=axis,top color=red!75!black,bottom color=white,shading angle=180] (O) -- (\arctandemi:4) arc (\arctandemi:90-\arctandemi:4) -- (O);
\shade [shading=axis,top color=white!75!black,bottom color=white,shading angle=180] (O) -- (90-\arctandemi:4) arc (90-\arctandemi:90:4) -- (O);

\draw[pointille,red] (-0.58,-0.5) -- (3.02,2.63) {};
\draw[pointille,red] (-0.5,-0.58) -- (2.63,3.02)  {};

\node[red] at (3,3) {$\ Q^-$};

\draw[axe,->] (O) -- (1,0) node[label=below right :{$\alpha=\rho_1$}] {};
\draw[axe,->] (O) -- (0,1) node[label=below left :{$\beta=\rho'_1$}] {};

\draw[->] (O) -- (2.02,1) node[label=right:{$\rho_2$}] {};
\draw[->] (O) -- (1,2.02) node[label=left:{$\rho'_2$}] {};

\draw[->] (O) -- (3.0804,2.02) node[label=right:{$\rho_3$}] {};
\draw[->] (O) -- (2.02,3.0804) node[label=left:{$\rho'_3$}] {};

\draw[->] (O) -- (4.202408,3.0804) node[label=right:{$\rho_4$}] {};
\draw[->] (O) -- (3.0804,4.202408) node[label=left:{$\rho'_4$}] {};

\draw[pointille,color=green!50!black] (-1,2) -- (2,-1) node[label=right:{$V_1$}] {};

\end{tikzpicture}

&

\begin{tikzpicture}
	[scale=1,
	 pointille/.style={densely dashed},
	 sommet/.style={inner sep=2pt,circle,draw=black,fill=blue,thick,anchor=base},
	 ]

\draw[pointille,color=green!50!black] (-2.5,0) -- (-0.1380,0);
\draw[pointille,color=green!50!black] (0.1380,0) -- (2.5,0) node[label=below right:{$V_1$}] {};

\coordinate (O) at (0,0);
\fill[red] (0.1380,0) circle (0.05);
\fill[red] (-0.1380,0) circle (0.05);

\draw[color=red] (-0.1380,0) -- (0.1380,0) {};

\fill (2,0) circle (0.05) node[label=above :{$\alpha=\rho_1$}] {}; 
\fill (-2,0) circle (0.05) node[label=above :{$\beta=\rho'_1$}] {};

\fill (0.675,0) circle (0.05) node[label=above:{$\widehat\rho_2$}] {}; 
\fill (-0.675,0) circle (0.05) node[label=above :{$\widehat\rho'_2$}] {};

\fill (0.4158,0) circle (0.05) ;
\fill (-0.4158,0) circle (0.05) ;

\fill (0.3081,0) circle (0.05) ;
\fill (-0.3081,0) circle (0.05) ;

\node[label=above:{$\cdots$}] at (O) {};
\draw[red] node[label=below:{\ $\widehat{Q^-}$}] at (O) {};

\coordinate (ancre) at (2.5,1);

\node[sommet,label=left:$s_\alpha$] (alpha) at (ancre) {};
\node[sommet,label=right:$s_\beta$] (beta) at ($(ancre)+(0.75,0)$) {} edge[thick] node[auto,swap] {$\infty(-1.01)$} (alpha);

\end{tikzpicture}
\end{tabular}}
\caption{The isotropic cone $Q$ and the first positive roots and normalized roots of an infinite based root system of rank $2$ with $B(\alpha,\beta)=-1.01$ in a Lorentzian $2$-space.}
\label{fig:dihed}
\end{figure}

In the Lorentzian case there is an explicit natural choice for a transverse cutting hyperplane: the affine hyperplane $\mathcal H$ (from \S\ref{ss:projmodel}), such that $Q\cap\mathcal H$ is a sphere. %% We assume\footnote{Otherwise we could restrict our study to the subspace $\Span(\Delta)$, in which $\Phi$ could be finite, affine or weakly hyperbolic, depending of the signature of the restriction of $B$ to this subspace, see \cite{DyHoRi13} for more details.} that $\Span(\Delta)=V$;

The discussion in~\S\ref{se:Hyperbo} depends heavily on the basis $\mathcal B$ of~\S\ref{ss:bases}. By~\cite[Proposition 4.13]{DyHoRi13}, we can fix $\mathcal B=(e_1,\dots,e_{n+1})$ as in~\S\ref{se:Hyperbo}  such that  $\mathcal H$ is transverse to~$\Phi^+$. Moreover, we may assume that $\mathcal B$ and $\mathcal H$ have the following properties:

\begin{enumerate}

\item $B(e_{n+1},\alpha)<0$ for all $\alpha\in \Delta$. Moreover $e_{n+1}$ lies in $\cone(\Delta)$ (by~\cite[Lemma 2.4, Proposition 4.13]{DyHoRi13}). This means that the fundamental tile $K$ of the  imaginary convex set (see Definition~\ref{def:imc}) contains the center of the ball model for $\mathbb H^n$.

\item  Denote by $H$ the linear hyperplane directing $\mathcal H$, and for $v\in V\setminus
H$, denote by $\widehat v$ the intersection point of the line $\mathbb Rv$
with $\mathcal H$ (we also use the analog notation~$\widehat P$ for a subset $P$ of
$V\setminus H$), see \cite[\S2.1 and \S5.2]{HoLaRi12} for more details. With these notations we have:

\begin{enumerate}
 
\item  $\mathbb{H}_p^n= \widehat{Q^-}=Q^-\cap \mathcal H$ and its boundary is $\partial \mathbb{H}_p^n= Q\cap \mathcal H=\widehat Q$,  by Proposition~\ref{prop:proj};

\item  for any $x\in \mathbb{H}^n_p$ and $w\in W$,  $w\cdot x = \widehat{w(x)}$, by Corollary~\ref{cor:action}.

\end{enumerate}
\end{enumerate}

\smallskip

\noindent In \cite{HoLaRi12,DyHoRi13}, the following objects were introduced and studied in the case of a general based root system:
\begin{itemize}
\item The set of {\em normalized roots $\widehat\Phi:=\{\widehat\beta\,|\,
    \beta\in \Phi\}$}, which is contained in the convex hull of the
  {\em normalized simple roots} $\widehat\alpha$ in $\widehat\Delta$, seen as   
  points in $\mathcal H$, the affine hyperplane $\big\{x_{n+1}=1\big\}$. The normalized roots are representatives of the
  directions of the roots, or in other words, of the roots seen in the
  projective space $\mathbb{P}V$. In~Figure~\ref{fig:intro},
  normalized roots are in blue, while the edges of the polytope
  $\conv(\widehat\Delta)$ are in green.

\item The set $E(\Phi)$ of accumulation points of $\widehat\Phi$, to which
  tends the blue shape in~Figure~\ref{fig:intro}. For short, we call
  the elements of $E(\Phi)$, which are limit points of normalized roots, the
  \emph{limit roots} of $\Phi$.
  
   \item The Coxeter group $W$ acts on $\conv(E(\Phi))\sqcup\widehat \Phi$ by the action $w\cdot x = \widehat{w(x)}$ above, see~\cite[\S2.3]{DyHoRi13}.

\end{itemize}

In the case the root system is weakly hyperbolic, the geometry of $E$ is well understood, and the action of $W$ on $E(\Phi)$ is faithful, see \cite[Theorem 6.1]{DyHoRi13}.

%%%%%%%%%%

\subsection{Convex hull of the limit roots and proof of Theorem~\ref{thm:main}}\label{ss:prooftheorem}

The set $E(\Phi)$ also enjoys some fractal properties as shown in \cite[\S4]{DyHoRi13}, see also \cite{rank3}. The main ingredient to explain these properties is its relation with Dyer's imaginary cone and the {\em imaginary convex set}, see \cite[\S2]{DyHoRi13}.  

\begin{defi}\label{def:imc}  The {\em imaginary convex set} $Z(\Phi)$ is the $W$-orbit of the polytope
\[
K :=\{v\in \conv(\widehat\Delta)\,|\, B(v,\alpha)\leq 0, \ \forall \alpha\in\Delta\}.
\]
Note that $K$ is the intersection of $\conv(\widehat\Delta)$ with the halfspaces $H_\alpha^-$ for $\alpha\in\Delta$, where $H_\alpha^-$ has for boundary the reflecting hyperplane $H_\alpha$, and does not contain $\widehat\alpha$.
\end{defi}

\begin{remark} The imaginary convex set $Z(\Phi)$ is the projective version of the {\em imaginary cone} that has been first introduced by Kac (see \cite[Ch.~5]{Kac90}) in the context of Weyl groups of Kac-Moody Lie algebras; this notion has been generalized afterwards to arbitrary Coxeter groups, first by H\'ee \cite{He93}, then by Dyer~\cite{Dy12} (see also Edgar's thesis \cite{Ed09} or Fu's article~\cite{Fu11-2}). The definition we use here is illustrated in Figure~\ref{fig:3}. This definition  is the specialization to Lorentzian spaces of a definition that applies to any {\em geometric representation} (over a quadratic space) of a  finitely generated Coxeter group; see \cite{DyHoRi13} for more details. The imaginary cone in Dyer's general definition is indeed a cone (\cite[Prop.~3.2.(b)]{Dy12}, and thus $Z(\Phi)$ is indeed a convex set.
\end{remark}

 The imaginary convex set, and its closure, are intimately linked with the set of limit roots, see~\cite[\S2]{DyHoRi13}:
\begin{itemize}
\item
$\conv(E(\Phi))=\overline{Z(\Phi)} \subseteq \{v\in \mathcal H\,|\, B(v,v)\leq 0\} =\overline{Q^-} \cap \mathcal H$.

\item In particular $K$ and $Z(\Phi)$ intersect $\widehat{Q^-}$ (which is exactly $\mathbb{H}_p^n$ in the Lorentzian case). 

\item $\conv(E(\Phi))=\overline{Z(\Phi)} $ is the unique non-empty closed $W$-invariant convex set contained
in $\conv(\Delta)$, see  \cite[Theorem 7.6]{Dy12}.

\end{itemize}

When $(\Phi,\Delta)$ is weakly hyperbolic, we have more information:
\begin{itemize}
\item by~\cite[Lemma 2.4]{DyHoRi13}, the polytope $K$ has a non-empty interior;
  \item the action of $W$ on the non-empty set $Z\cap \mathbb{H}_p^n$ is the same as the action from Corollary~\ref{cor:action};
\item by~\cite[Theorem~4.10]{DyHoRi13}, we have  $E(\Phi)=\overline Z\cap Q$, 
  and by~\cite[Corollary 6.15(c)]{DyHoRi13} the set of limit roots is also the accumulation set of any orbit in $Z$:
\begin{equation}
\label{eq:fund} 
E(\Phi)=\textrm{Acc}(W\cdot v),\quad \textrm{for any }\ v\in Z.
\end{equation}
\end{itemize}

\begin{figure}[!h]
\begin{minipage}[b]{\linewidth}
\centering
\scalebox{0.8}{\begin{tikzpicture}
	[scale=2,
	 q/.style={red,thick},
	 racine/.style={blue},
	 racinesimple/.style={black},
	 racinedih/.style={blue},
	 sommet/.style={inner sep=2pt,circle,draw=black,fill=blue,thick,anchor=base},
	 rotate=0]

\def\grosseur{0.0125}
\def\grosseursimple{0.025}

\def\grosseurdih{0.0075}

%% La courbe Q

% ici Q est un cercle, le centre est le centre de gravité du triangle
% (2,2/sqrt(3)), et le rayon se calcule : si tous les inner product
% valent -k, on trouve sauf erreur 
% r = sqrt( 4/3 + 4(k^2-1)/(k+1)^2 )

\draw[q] (2, 1.15470053837925) circle (1.30267789455786) ;

% Roots of deepness= 1
\node[label=left :{$\alpha$}] (a) at (0.000000000000000,0.000000000000000) {};
\fill[racinesimple] (0.000000000000000,0.000000000000000) circle (\grosseursimple);\node[label=right :{$\beta$}] (b) at (4.00000000000000,0.000000000000000) {};
\fill[racinesimple] (4.00000000000000,0.000000000000000) circle (\grosseursimple);\node[label=above :{$\gamma$}] (g) at (2.00000000000000,3.46410161513775) {};
\fill[racinesimple] (2.00000000000000,3.46410161513775) circle (\grosseursimple);
\draw[green!75!black] (a) -- (b) -- (g) -- (a);
% Roots of deepness= 2
\fill[racine] (1.17647058823529,0.0100000000000000) circle (\grosseursimple);
\fill[racine] (1.41176470588235,2.44524819892077) circle (\grosseursimple);
\fill[racine] (2.82352941176471,0.0100000000000000) circle (\grosseursimple);
\fill[racine] (3.41176470588235,1.01885341621699) circle (\grosseursimple);
\fill[racine] (0.588235294117647,1.01885341621699) circle (\grosseursimple);
\fill[racine] (2.58823529411765,2.44524819892077) circle (\grosseursimple);

\coordinate (ancre) at (-0.5,2.6);
\node[sommet,label=below left:$s_\alpha$] (alpha) at (ancre) {};
\node[sommet,label=below right :$s_\beta$] (beta) at ($(ancre)+(0.5,0)$) {} edge[thick] node[auto] {-1.2} (alpha);
\node[sommet,label=above:$s_\gamma$] (gamma) at ($(ancre)+(0.25,0.43)$) {} edge[thick] node[auto,swap] {-1.2} (alpha) edge[thick] node[auto] {-1.2} (beta);

%% Drawing K 
\filldraw[draw= black ,fill= yellow ,opacity= 0.300000000000000 ]
(2.909090908, 1.8895099722) --
(1.0909090908, 1.8895099722) --
(0.909090909, 1.5745916434) --
(1.8181818182, 0.0) --
(2.181818182, 0.0) --
(3.09090909, 1.5745916426) --
cycle ;
% Coordinates of gravity center of this polygon:
% (1.99999999967, 1.1547005384)

%% Drawing iterated images of K of step 1 

\filldraw[draw= black ,fill= yellow!80 ,opacity= 0.300000000000000 ]
(0.855614973194145, 0.555738227281955)  -- 
(0.779220779101243, 1.34964998007078)  -- 
(0.909090909001093, 1.57459164340189)  -- 
(1.81818181816364, 0.000000000000000)  -- 
(1.55844155834879, 0.000000000000000)  -- 
(0.909090909200340, 0.463115189191957)  -- 
cycle ;
% Coordinates of gravity center of this polygon:
% (1.13827349116821, 0.657182506657764)

\filldraw[draw= black ,fill= yellow!80 ,opacity= 0.300000000000000 ]
(3.22077922077631, 1.34964997879825)  -- 
(3.14438502678509, 0.555738227066197)  -- 
(3.09090909093551, 0.463115189208216)  -- 
(2.44155844154916, 0.000000000000000)  -- 
(2.18181818163636, 0.000000000000000)  -- 
(3.09090909053703, 1.57459164166984)  -- 
cycle ;
% Coordinates of gravity center of this polygon:
% (2.86172650870324, 0.657182506123751)

\filldraw[draw= black ,fill= yellow!80 ,opacity= 0.300000000000000 ]
(2.90909090835412, 1.88950997158664)  -- 
(1.09090909044588, 1.88950997158664)  -- 
(1.22077922060431, 2.11445163513458)  -- 
(1.94652406417647, 2.44524819892077)  -- 
(2.05347593588235, 2.44524819892077)  -- 
(2.77922077811584, 2.11445163554275)  -- 
cycle ;
% Coordinates of gravity center of this polygon:
% (1.99999999959649, 2.14973660194869)

%% Drawing iterated images of K of step 2 

\filldraw[draw= black ,fill= yellow!60 ,opacity= 0.300000000000000 ]
(2.88170407005843, 0.197647486673624)  -- 
(3.05271199391595, 0.396955876465156)  -- 
(3.09090909093583, 0.463115189208773)  -- 
(2.44155844156772, 0.000000000000000)  -- 
(2.55161787368442, 0.000000000000000)  -- 
(2.85271933305869, 0.171898651017698)  -- 
cycle ;
% Coordinates of gravity center of this polygon:
% (2.81187013387017, 0.204936200560875)

\filldraw[draw= black ,fill= yellow!60 ,opacity= 0.300000000000000 ]
(1.57522826510843, 2.38457808676480)  -- 
(1.27580893667477, 2.20976589928114)  -- 
(1.22077922060375, 2.11445163513362)  -- 
(1.94652406416578, 2.44524819892077)  -- 
(1.87012987010259, 2.44524819892077)  -- 
(1.61201977942229, 2.39680518749896)  -- 
cycle ;
% Coordinates of gravity center of this polygon:
% (1.58341502267960, 2.33268286775334)

\filldraw[draw= black ,fill= yellow!60 ,opacity= 0.300000000000000 ]
(0.947288006310731, 0.396955876200963)  -- 
(1.11829592999777, 0.197647486608127)  -- 
(1.14728066692370, 0.171898651013539)  -- 
(1.44838212635149, 0.000000000000000)  -- 
(1.55844155853432, 0.000000000000000)  -- 
(0.909090909358289, 0.463115188918381)  -- 
cycle ;
% Coordinates of gravity center of this polygon:
% (1.18812986624605, 0.204936200456835)

\filldraw[draw= black ,fill= yellow!60 ,opacity= 0.300000000000000 ]
(2.72419106255821, 2.20976589972896)  -- 
(2.42477173484066, 2.38457808679453)  -- 
(2.38798022062887, 2.39680518749691)  -- 
(2.12987012986740, 2.44524819892077)  -- 
(2.05347593577540, 2.44524819892077)  -- 
(2.77922077784185, 2.11445163601732)  -- 
cycle ;
% Coordinates of gravity center of this polygon:
% (2.41658497691873, 2.33268286797987)

\filldraw[draw= black ,fill= yellow!60 ,opacity= 0.300000000000000 ]
(0.855614973298298, 0.555738227101555)  -- 
(0.779220779281918, 1.34964998038372)  -- 
(0.724191063174665, 1.25433571593848)  -- 
(0.722508932116326, 0.907624877326688)  -- 
(0.730315709398947, 0.869648941020573)  -- 
(0.817417876158433, 0.621897540049636)  -- 
cycle ;
% Coordinates of gravity center of this polygon:
% (0.771544888904765, 0.926482546970108)

\filldraw[draw= black ,fill= yellow!60 ,opacity= 0.300000000000000 ]
(3.22077922059563, 1.34964997911119)  -- 
(3.14438502668094, 0.555738226885798)  -- 
(3.18258212374129, 0.621897539719493)  -- 
(3.26968429060517, 0.869648941040586)  -- 
(3.27749106788816, 0.907624877348487)  -- 
(3.27580893678228, 1.25433571494016)  -- 
cycle ;
% Coordinates of gravity center of this polygon:
% (3.22845511104891, 0.926482546507618)

%% Noms des regions:
\draw  (2, 1.15) node{$K$};
\draw  (1.13827349116821, 0.657182506657764) node{{\small $s_{\alpha}\!\cdot\!K$}};

\end{tikzpicture}}
\end{minipage}%
\caption{Picture in rank $3$ of the first $9$ normalized roots  for the weakly hyperbolic root system with diagram in the upper left of the picture. In red is $\widehat Q=\partial \mathbb{H}_p^2$ and in green is the boundary of  $\conv(\widehat\Delta)$. Some reflection hyperplanes for $W$, $K$ and some of its images under reflections are also represented in shaded yellow.}
\label{fig:3}
\end{figure}
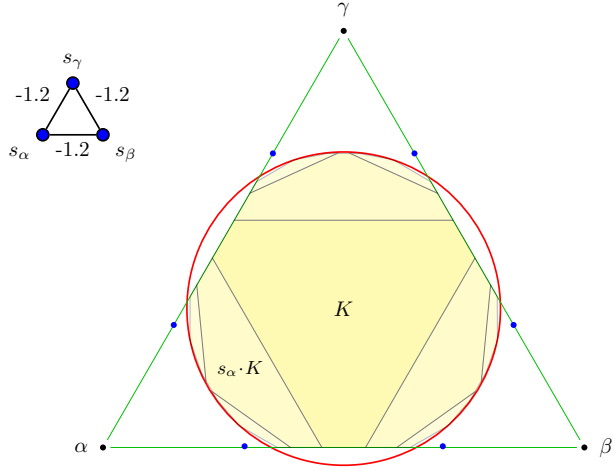

\begin{remark} The property that $E(\Phi)=\overline Z\cap Q$ is crucial to prove Equation~\eqref{eq:fund}. Whether these two properties are true or not for other irreducible (i.e. the Coxeter graph is connected) root systems (when $(V,B)$ is not  a Lorentzian space) is an open question, see \cite[Question~4.9]{DyHoRi13} and the proof of \cite[Corollary 6.15(c)]{DyHoRi13}.
\end{remark}

We prove now Theorem~\ref{thm:main} from the introduction, whose precise statement is given below.

\begin{thm}\label{thm:main2} Let $(\Phi,\Delta)$ be a weakly hyperbolic root system in a Lorentzian space $(V,B)$, and $W\subseteq \mathcal I(\mathbb{H}_p^n)$ its associated Coxeter group. Then the limit set $\Lambda(W)$ of~$W$ is equal to the set $E(\Phi)$ of limit roots of $\Phi$. 
\end{thm}
\begin{proof}  To study the limit set of~$W$ we may choose any point in $\mathbb{H}_p^n$ (see \S\ref{ss:LimitSet}), so it satisfies:
\[
\Lambda(W)=\textrm{Acc}(W\cdot v),
\]
for any $v\in Z\cap \mathbb{H}_p^n$. Since $Z$ is non-empty, Theorem~\ref{thm:main2}, and therefore Theorem~\ref{thm:main}, follows from Equation~\ref{eq:fund}.
\end{proof}

\subsection{Discrete groups generated by hyperbolic reflections and Coxeter groups}  The relation between ``hyperbolic Coxeter groups'' and discrete groups generated by hyperbolic reflections is not always transparent in the literature. The terminology ``hyperbolic Coxeter group'' is used by many
authors, but not in a consistent way.  For instance, in the article of Krammer~\cite{Kr09}, ``hyperbolic Coxeter group'' means a Coxeter group attached to a weakly hyperbolic root system, whereas in Humphreys' book~\cite{Hu90} this means a strict subclass of Coxeter groups attached to a weakly hyperbolic root system. See also the difference in the use of these expressions between~\cite{Da08,Ra06,AbBr08} or \cite{Do08}.  We  end this section by clarifying the relation between those terms.

\subsubsection{Discrete reflection groups of hyperbolic isometries}\label{see:Vinberg}  A {\em discrete reflection group on $\mathbb{H}^n$} is  a discrete subgroup  of hyperbolic isometries $\Gamma  \subseteq \mathcal I(\mathbb{H}^n)$ generated by (finitely many) hyperbolic reflections, see \cite[Chapter 5, \S1.2]{Vi93}. We explained before how  a Coxeter group with a root system in a Lorentzian space has a representation as a discrete reflection group, see Proposition~\ref{prop:repHypo}. Conversely, we have the following theorem.

\begin{thm}[Vinberg~\cite{Vi71,Vi93}]\label{thmVinberg} Discrete reflection groups on $\mathbb H^n$ are Coxeter groups that are  associated to  based root systems in Lorentzian spaces. 
\end{thm}

This result due to Vinberg completes for spaces of constant curvature the classical result of Coxeter~\cite{Co34} that  shows that: (1) discrete reflection groups on the sphere  (equivalently on the Euclidean vector space)  are finite and are Coxeter groups; (2)  discrete reflection groups on  affine Euclidean space are Coxeter groups. Moreover, Coxeter classified those groups in the case of the sphere and of the affine Euclidean space, see \cite{Hu90}. Such a classification is still incomplete for Coxeter groups that arise as discrete reflection groups on $\mathbb{H}^n$. Only the subclass of {\em hyperbolic reflection groups} is classified, see below for more details.

\begin{remark}[Remark on the proof of Theorem \ref{thmVinberg}] To show that a discrete reflection group~$\Gamma$ on~$\mathbb{H}^n$ is a Coxeter group is a bit more complicated than in the case of the sphere and of the affine Euclidean space, but the main steps are the same. To our knowledge, the theorem above is never stated precisely in these terms, but is clearly apparent in \cite[Chapter 5, \S1.1 and \S1.2]{Vi93}.

First, consider the hyperplane arrangement associated to $\Gamma$, i.e., the set of hyperplanes of the reflections in $\Gamma$. This hyperplane arrangement is locally finite and therefore decomposes $\mathbb H^n$ into (hyperbolic) convex polyhedra (each of these is a {\em Dirichlet domain}, see \cite{Vi93} or \cite[\S6.6]{Ra06}). Pick one of these to be the {\em fundamental chamber}\footnote{Also called {\em fundamental convex polyhedron} in the literature, see \cite{Ra06,Vi93}} $P$, which is a convex polyhedron. Then this fundamental chamber is a {\em fundamental domain} for the action of $\Gamma$ on $\mathbb H^n$ and the angles between the facets that intersect\footnote{See Remark~\ref{prop:fundamentallink}.} are {\em submultiples of $\pi$}, see \cite[Chapter 5, \S1.1 and Proposition 1.4]{Vi93} or \cite[Theorem 2.1]{Do08}.  Then, by \cite[Theorem 2]{Vi71}, we know that the exterior unitary (for $B$) normal vectors in $V$ associated to the facets of $P$  form a simple system $\Delta$ and the union $\Phi$ of their orbits is a root system. Finally, by denoting $S$  the set of reflections through the facets of $P$, one deduces that $(\Gamma, S)$ is a Coxeter system associated to the based root system $(\Phi,\Delta)$ (see also \cite[Theorem 1.2.2]{Kr09}).
\end{remark}

\subsubsection{Fundamental polyhedron of discrete reflection groups of hyperbolic isometries} Let $W\subseteq \mathcal I(\mathbb{H}^n)$ be a discrete reflection group on $\mathbb{H}^n$, with associated root system $(\Phi,\Delta)$ in $(V,B)$. A fundamental convex polyhedron $P$, i.e., a fundamental chamber as in the remark above,  can be easily described with the help of the {\em Tits cone}. Since $B$ is non-degenerate, we associate $V$ and its dual. So  
\[
C :=\{v\in V\,|\, B(v,\alpha)\geq 0, \ \forall \alpha\in\Delta\}
\] 
is a fundamental domain for the action of $W$  on the {\em Tits cone} $\mathcal U$, which is the union of the $W$-orbit $W(C)$; we call $C$ the 
{\em fundamental chamber}. 

\begin{remark}
Observe that with Definition~\ref{def:imc}, $K=(-C)\cap\conv(\widehat\Delta)$, and the imaginary convex set $Z$  is contained in the negative of the Tits cone~$-\mathcal U$, see Figure~\ref{fig:3} or Figure~\ref{fig:4}.
\end{remark}

The following proposition is mostly a reinterpretation of classical results, see for instance~\cite{Kr09}.

 \begin{prop}\label{prop:Tits} Let $(\Phi,\Delta)$ be a root system in the Lorentzian $(n+1)$-space $(V,B)$ with associated Coxeter system $(W,S)$.  Then 
 \begin{enumerate}[(i)]
 \item $(-\mathcal U)\cap \mathbb H^n=\mathbb H^n$ and $(-\mathcal U)\cap \mathbb{H}_p^n=\mathbb{H}_p^n$. 
\item  $P:=(-C)\cap \mathbb H^n$ is a fundamental domain for the action of the discrete group of isometries $W$ on $\mathbb H^n$.
\item  $(-C)\cap  \mathbb{H}_p^n$ is a fundamental domain for the action of the discrete group of isometries $W$ on $\mathbb{H}_p^n$ described in Corollary~\ref{cor:action}.
  \end{enumerate}
\end{prop}
\begin{proof} By~\cite[Proposition 4.6.1]{Kr09}, we know that either $\mathcal U\cap Q^-=\{v\in Q^-\,|\, x_{n+1}<0\}$ or $\mathcal U \cap Q^-=\{v\in Q^-\,|\, x_{n+1}>0\}$.
We have $B(e_{n+1},\alpha)<0$ for all $\alpha\in \Delta$, by item (1) of \S\ref{ss:limitroots}. Thus we have $e_{n+1}\in -C$.  So~$(-\mathcal U)\cap Q^- =\{v\in Q^-\,|\, x_{n+1}>0\}$.
 Therefore $(-\mathcal U)\cap \mathbb H^n=\mathbb H^n$ and $(-\mathcal U)\cap \mathbb{H}_p^n=\mathbb{H}_p^n$.  The rest of the proposition follows from the fact that $-C$ is a fundamental domain for $-\mathcal U$.
\end{proof}

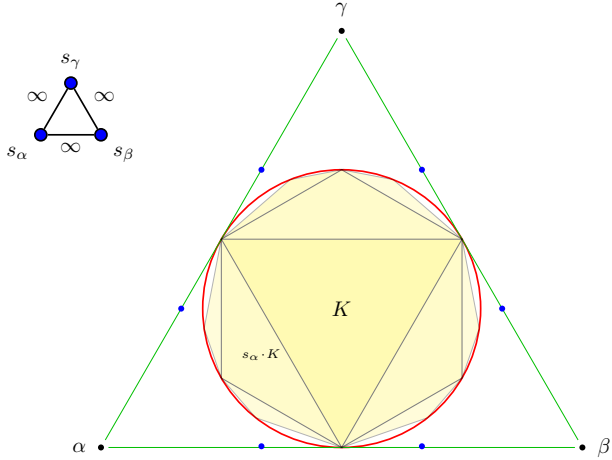
\begin{figure}[!h]
\begin{minipage}[b]{\linewidth}
\centering
\scalebox{0.8}{\begin{tikzpicture}
	[scale=2,
	 q/.style={red,line join=round,thick},
	 racine/.style={blue},
	 racinesimple/.style={black},
	 racinedih/.style={blue},
	 sommet/.style={inner sep=2pt,circle,draw=black,fill=blue,thick,anchor=base},
	 rotate=0]

\def\grosseur{0.0125}
\def\grosseursimple{0.025}

\def\grosseurdih{0.0075}

%% La courbe Q

% ici Q est un cercle, le centre est le centre de gravité du triangle
% (2,2/sqrt(3)), et le rayon se calcule : si tous les inner product
% valent -k, on trouve sauf erreur 
% r = sqrt( 4/3 + 4(k^2-1)/(k+1)^2 )

\draw[q] (2, 1.15470053837925) circle (1.15470053837925) ;

% Roots of deepness= 1
\node[label=left :{$\alpha$}] (a) at (0.000000000000000,0.000000000000000) {};
\fill[racinesimple] (0.000000000000000,0.000000000000000) circle (\grosseursimple);\node[label=right :{$\beta$}] (b) at (4.00000000000000,0.000000000000000) {};
\fill[racinesimple] (4.00000000000000,0.000000000000000) circle (\grosseursimple);\node[label=above :{$\gamma$}] (g) at (2.00000000000000,3.46410161513775) {};
\fill[racinesimple] (2.00000000000000,3.46410161513775) circle (\grosseursimple);
\draw[green!75!black] (a) -- (b) -- (g) -- (a);
% Roots of deepness= 2
\fill[racine] (1.33333333333333,0.0100000000000000) circle (\grosseursimple);
\fill[racine] (2.66666666666667,2.30940107675850) circle (\grosseursimple);
\fill[racine] (3.33333333333333,1.15470053837925) circle (\grosseursimple);
\fill[racine] (0.666666666666667,1.15470053837925) circle (\grosseursimple);
\fill[racine] (2.66666666666667,0.0100000000000000) circle (\grosseursimple);
\fill[racine] (1.33333333333333,2.30940107675850) circle (\grosseursimple);

\coordinate (ancre) at (-0.5,2.6);
\node[sommet,label=below left:$s_\alpha$] (alpha) at (ancre) {};
\node[sommet,label=below right :$s_\beta$] (beta) at ($(ancre)+(0.5,0)$) {} edge[thick] node[auto] {$\infty$} (alpha);
\node[sommet,label=above:$s_\gamma$] (gamma) at ($(ancre)+(0.25,0.43)$) {} edge[thick] node[auto,swap] {$\infty$} (alpha) edge[thick] node[auto] {$\infty$} (beta);

%% Drawing K 
\filldraw[draw= black ,fill= yellow ,opacity= 0.300000000000000 ]
(0.9999999998, 1.7320508078) --
(2.0, 0.0) --
(3.0, 1.7320508082) --
cycle ;
% Coordinates of gravity center of this polygon:
% (1.99999999993, 1.15470053867)

%% Drawing iterated images of K of step 1 

\filldraw[draw= black ,fill= yellow!80 ,opacity= 0.300000000000000 ]
(0.999999999866561, 1.73205080791529)  -- 
(2.00000000000000, 0.000000000000000)  -- 
(0.999999999878540, 0.577350269329875)  -- 
cycle ;
% Coordinates of gravity center of this polygon:
% (1.33333333324837, 0.769800359081721)

\filldraw[draw= black ,fill= yellow!80 ,opacity= 0.300000000000000 ]
(3.00000000004448, 0.577350269202496)  -- 
(2.00000000000000, 0.000000000000000)  -- 
(3.00000000036438, 1.73205080756888)  -- 
cycle ;
% Coordinates of gravity center of this polygon:
% (2.66666666680295, 0.769800358923791)

\filldraw[draw= black ,fill= yellow!80 ,opacity= 0.300000000000000 ]
(0.999999999533123, 1.73205080733775)  -- 
(2.00000000000000, 2.30940107675850)  -- 
(3.00000000072876, 1.73205080693775)  -- 
cycle ;
% Coordinates of gravity center of this polygon:
% (2.00000000008729, 1.92450089701134)

%% Drawing iterated images of K of step 2 

\filldraw[draw= black ,fill= yellow!60 ,opacity= 0.300000000000000 ]
(3.00000000006667, 0.577350269240926)  -- 
(2.00000000000000, 0.000000000000000)  -- 
(2.71428571434520, 0.247435829691337)  -- 
cycle ;
% Coordinates of gravity center of this polygon:
% (2.57142857147062, 0.274928699644087)

\filldraw[draw= black ,fill= yellow!60 ,opacity= 0.300000000000000 ]
(0.999999999466562, 1.73205080722247)  -- 
(2.00000000000000, 2.30940107675850)  -- 
(1.57142857134677, 2.22692246684851)  -- 
cycle ;
% Coordinates of gravity center of this polygon:
% (1.52380952360444, 2.08945811694316)

\filldraw[draw= black ,fill= yellow!60 ,opacity= 0.300000000000000 ]
(1.28571428570474, 0.247435829652708)  -- 
(2.00000000000000, 0.000000000000000)  -- 
(1.00000000000000, 0.577350269119501)  -- 
cycle ;
% Coordinates of gravity center of this polygon:
% (1.42857142856825, 0.274928699590736)

\filldraw[draw= black ,fill= yellow!60 ,opacity= 0.300000000000000 ]
(2.42857142859322, 2.22692246687191)  -- 
(2.00000000000000, 2.30940107675850)  -- 
(3.00000000036438, 1.73205080756888)  -- 
cycle ;
% Coordinates of gravity center of this polygon:
% (2.47619047631920, 2.08945811706643)

\filldraw[draw= black ,fill= yellow!60 ,opacity= 0.300000000000000 ]
(1.00000000013344, 1.73205080837753)  -- 
(0.857142857142857, 0.989743318610787)  -- 
(1.00000000012146, 0.577350268909127)  -- 
cycle ;
% Coordinates of gravity center of this polygon:
% (0.952380952465918, 1.09971479863248)

\filldraw[draw= black ,fill= yellow!60 ,opacity= 0.300000000000000 ]
(2.99999999995552, 0.577350269048415)  -- 
(3.14285714285714, 0.989743318610787)  -- 
(2.99999999963562, 1.73205080883112)  -- 
cycle ;
% Coordinates of gravity center of this polygon:
% (3.04761904748276, 1.09971479883011)

%% Noms des regions:
\draw (2,1.15) node{$K$};
\draw (1.33, 0.77) node{{\tiny $s_{\alpha}\!\cdot\!K$}};
\end{tikzpicture}}
\end{minipage}%
\caption{A {\em hyperbolic root system.} Picture in rank $3$ of the first normalized roots converging toward the set of limits of roots $E(\Phi)$ for the root system with diagram in the upper left of the picture. In red is $\widehat Q=\partial \mathbb{H}_p^2$. Reflection hyperplanes for $W$ and $K$ are also represented. Note that here $K\cap \mathbb{H}_p^2=-C\cap \mathbb{H}^2_p$ is not compact but has a finite volume in $\mathbb{H}_p^2$, see~\S\ref{sss:hyper}.}
\label{fig:4}
\end{figure}

\subsubsection{Hyperbolic Coxeter groups}\label{sss:hyper}
We define two strict subclasses of weakly hyperbolic root systems, after  \cite[\S9]{Dy12} or \cite[\S4.1]{DyHoRi13}. Let $W$ be a discrete reflection group on $\mathbb{H}^n$, and let $(\Phi,\Delta)$ be its associated (weakly hyperbolic) based root system. We say that
\begin{itemize}
\item $W$ is {\em hyperbolic}, and $(\Phi,\Delta)$ is a {\em hyperbolic} root system, if and only if $B$ restricts to a positive form on each proper face of $\cone(\Delta)$, i.e. $-C\subseteq \conv(\widehat\Delta)$ (an example is given in Figure~\ref{fig:4});
\item $W$ is {\em compact hyperbolic}, and $(\Phi,\Delta)$ is a {\em compact hyperbolic} root system, if and only if $B$ restricts to a positive definite form on each proper face of $\cone(\Delta)$, i.e. $-C\subseteq \textrm{int}(\conv(\widehat \Delta))$ (an example is given in Figure~\ref{fig:5}). 
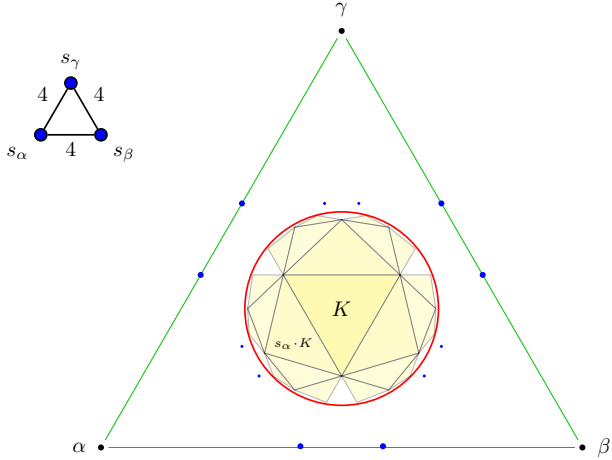
\begin{figure}[!h]
\begin{minipage}[b]{\linewidth}
\centering
\scalebox{0.8}{\begin{tikzpicture}
	[scale=2,
	 q/.style={red,line join=round,thick},
	 racine/.style={blue},
	 racinesimple/.style={black},
	 racinedih/.style={blue},
	 sommet/.style={inner sep=2pt,circle,draw=black,fill=blue,thick,anchor=base},
	 rotate=0]

\def\grosseur{0.0125}
\def\grosseursimple{0.025}

\def\grosseurdih{0.0075}

%% La courbe Q

%% La courbe Q

% ici Q est un cercle, le centre est le centre de gravité du triangle
% (2,2/sqrt(3)), et le rayon se calcule : si tous les inner product
% valent -k, on trouve sauf erreur 
% r = sqrt( 4/3 + 4(k^2-1)/(k+1)^2 )

\draw[q] (2, 1.15470053837925) circle (0.804389105046863) ;

% Roots of deepness= 1
\node[label=left :{$\alpha$}] (a) at (0.000000000000000,0.000000000000000) {};
\fill[racinesimple] (0.000000000000000,0.000000000000000) circle (\grosseursimple);\node[label=right :{$\beta$}] (b) at (4.00000000000000,0.000000000000000) {};
\fill[racinesimple] (4.00000000000000,0.000000000000000) circle (\grosseursimple);\node[label=above :{$\gamma$}] (g) at (2.00000000000000,3.46410161513775) {};
\fill[racinesimple] (2.00000000000000,3.46410161513775) circle (\grosseursimple);
\draw[green!75!black] (a) -- (b) -- (g) -- (a);
% Roots of deepness= 2
\fill[racine] (0.828427124746190,1.43487787042860) circle (\grosseursimple);
\fill[racine] (2.34314575050762,0.0100000000000000) circle (\grosseursimple);
\fill[racine] (1.17157287525381,2.02922374470915) circle (\grosseursimple);
\fill[racine] (3.17157287525381,1.43487787042860) circle (\grosseursimple);
\fill[racine] (1.65685424949238,0.0100000000000000) circle (\grosseursimple);
\fill[racine] (2.82842712474619,2.02922374470915) circle (\grosseursimple);

% Roots of deepness= 3
\fill[racine] (2.14213562373095,2.02922374470915) circle (\grosseur);
\fill[racine] (1.85786437626905,2.02922374470915) circle (\grosseur);
\fill[racine] (2.82842712474619,0.840531996148050) circle (\grosseur);
\fill[racine] (1.31370849898476,0.594345874280551) circle (\grosseur);
\fill[racine] (2.68629150101524,0.594345874280551) circle (\grosseur);
\fill[racine] (1.17157287525381,0.840531996148050) circle (\grosseur);

\coordinate (ancre) at (-0.5,2.6);
\node[sommet,label=below left:$s_\alpha$] (alpha) at (ancre) {};
\node[sommet,label=below right :$s_\beta$] (beta) at ($(ancre)+(0.5,0)$) {} edge[thick] node[auto] {4} (alpha);
\node[sommet,label=above:$s_\gamma$] (gamma) at ($(ancre)+(0.25,0.43)$) {} edge[thick] node[auto,swap] {4} (alpha) edge[thick] node[auto] {4} (beta);

%% Drawing K 
\filldraw[draw= black ,fill= yellow ,opacity= 0.300000000000000 ]
(2.485281374, 1.4348778704) --
(2.0, 0.5943458742) --
(1.5147186258, 1.4348778704) --
cycle ;
% Coordinates of gravity center of this polygon:
% (1.99999999993, 1.15470053833)

%% Drawing iterated images of K of step 1 

\filldraw[draw= black ,fill= yellow!80 ,opacity= 0.300000000000000 ]
(1.35924551799730, 0.784760765770221)  -- 
(2.00000000007939, 0.594345874223593)  -- 
(1.51471862577148, 1.43487787037299)  -- 
cycle ;
% Coordinates of gravity center of this polygon:
% (1.62465471461606, 0.937994836788933)

\filldraw[draw= black ,fill= yellow!80 ,opacity= 0.300000000000000 ]
(2.48528137428710, 1.43487787012804)  -- 
(1.99999999992061, 0.594345874223593)  -- 
(2.64075448202022, 0.784760765696952)  -- 
cycle ;
% Coordinates of gravity center of this polygon:
% (2.37534528540931, 0.937994836682860)

\filldraw[draw= black ,fill= yellow!80 ,opacity= 0.300000000000000 ]
(2.48528137398632, 1.43487787045720)  -- 
(2.00000000000000, 1.89458008377702)  -- 
(1.51471862581368, 1.43487787045720)  -- 
cycle ;
% Coordinates of gravity center of this polygon:
% (1.99999999993333, 1.58811194156381)

%% Drawing iterated images of K of step 2 

\filldraw[draw= black ,fill= yellow!60 ,opacity= 0.300000000000000 ]
(2.39052429178173, 0.478292623520252)  -- 
(1.99999999988772, 0.594345874280551)  -- 
(2.64075448201283, 0.784760765677444)  -- 
cycle ;
% Coordinates of gravity center of this polygon:
% (2.34375959122743, 0.619133087826082)

\filldraw[draw= black ,fill= yellow!60 ,opacity= 0.300000000000000 ]
(1.60947570824960, 1.83110845325340)  -- 
(2.00000000004342, 1.89458008376996)  -- 
(1.51471862579808, 1.43487787048422)  -- 
cycle ;
% Coordinates of gravity center of this polygon:
% (1.70806477803037, 1.72018880250253)

\filldraw[draw= black ,fill= yellow!60 ,opacity= 0.300000000000000 ]
(1.35924551807178, 0.784760765573826)  -- 
(2.00000000011228, 0.594345874280551)  -- 
(1.60947570824969, 0.478292623481762)  -- 
cycle ;
% Coordinates of gravity center of this polygon:
% (1.65624040881125, 0.619133087778713)

\filldraw[draw= black ,fill= yellow!60 ,opacity= 0.300000000000000 ]
(2.48528137414334, 1.43487787072917)  -- 
(1.99999999995658, 1.89458008376996)  -- 
(2.39052429174389, 1.83110845328062)  -- 
cycle ;
% Coordinates of gravity center of this polygon:
% (2.29193522194794, 1.72018880259325)

\filldraw[draw= black ,fill= yellow!60 ,opacity= 0.300000000000000 ]
(1.35924551797754, 0.784760765794419)  -- 
(1.21895141648864, 1.15470053838558)  -- 
(1.51471862572478, 1.43487787037299)  -- 
cycle ;
% Coordinates of gravity center of this polygon:
% (1.36430518673032, 1.12477972485099)

\filldraw[draw= black ,fill= yellow!60 ,opacity= 0.300000000000000 ]
(2.48528137433380, 1.43487787012804)  -- 
(2.78104858351136, 1.15470053838558)  -- 
(2.64075448203998, 0.784760765721150)  -- 
cycle ;
% Coordinates of gravity center of this polygon:
% (2.63569481329505, 1.12477972474492)

%% Drawing iterated images of K of step 3 

\filldraw[draw= black ,fill= yellow!60 ,opacity= 0.300000000000000 ]
(1.87324262455499, 0.374795659876179)  -- 
(2.00000000007939, 0.594345874337509)  -- 
(1.60947570826081, 0.478292623474515)  -- 
cycle ;
% Coordinates of gravity center of this polygon:
% (1.82757277763173, 0.482478052562734)

\filldraw[draw= black ,fill= yellow!60 ,opacity= 0.300000000000000 ]
(2.20101012679684, 1.92725011415193)  -- 
(1.99999999993859, 1.89458008375292)  -- 
(2.39052429173480, 1.83110845328786)  -- 
cycle ;
% Coordinates of gravity center of this polygon:
% (2.19717813949008, 1.88431288373091)

\filldraw[draw= black ,fill= yellow!60 ,opacity= 0.300000000000000 ]
(1.26120387497776, 1.43487787042113)  -- 
(1.21895141649015, 1.15470053835764)  -- 
(1.51471862570918, 1.43487787040000)  -- 
cycle ;
% Coordinates of gravity center of this polygon:
% (1.33162463905903, 1.34148542639292)

\filldraw[draw= black ,fill= yellow!60 ,opacity= 0.300000000000000 ]
(2.76955262169143, 0.942505626685999)  -- 
(2.78104858350771, 1.15470053840981)  -- 
(2.64075448205124, 0.784760765724497)  -- 
cycle ;
% Coordinates of gravity center of this polygon:
% (2.73045189575013, 0.960655643606768)

\filldraw[draw= black ,fill= yellow!60 ,opacity= 0.300000000000000 ]
(2.39052429166983, 0.478292623447298)  -- 
(1.99999999992061, 0.594345874337509)  -- 
(2.12675737540241, 0.374795659859466)  -- 
cycle ;
% Coordinates of gravity center of this polygon:
% (2.17242722233095, 0.482478052548091)

\filldraw[draw= black ,fill= yellow!60 ,opacity= 0.300000000000000 ]
(1.60947570834107, 1.83110845332635)  -- 
(2.00000000006141, 1.89458008375292)  -- 
(1.79898987322326, 1.92725011416213)  -- 
cycle ;
% Coordinates of gravity center of this polygon:
% (1.80282186054191, 1.88431288374713)

\filldraw[draw= black ,fill= yellow!60 ,opacity= 0.300000000000000 ]
(1.35924551786415, 0.784760765828115)  -- 
(1.21895141649229, 1.15470053840981)  -- 
(1.23044737830025, 0.942505626696196)  -- 
cycle ;
% Coordinates of gravity center of this polygon:
% (1.26954810421890, 0.960655643644707)

\filldraw[draw= black ,fill= yellow!60 ,opacity= 0.300000000000000 ]
(2.48528137449082, 1.43487787040000)  -- 
(2.78104858350985, 1.15470053835764)  -- 
(2.73879612503588, 1.43487787042113)  -- 
cycle ;
% Coordinates of gravity center of this polygon:
% (2.66837536101218, 1.34148542639292)

\filldraw[draw= black ,fill= yellow!60 ,opacity= 0.300000000000000 ]
(2.39052429179793, 0.478292623526608)  -- 
(2.56854249492711, 0.594345874280551)  -- 
(2.64075448201692, 0.784760765660333)  -- 
cycle ;
% Coordinates of gravity center of this polygon:
% (2.53327375624732, 0.619133087822497)

\filldraw[draw= black ,fill= yellow!60 ,opacity= 0.300000000000000 ]
(1.60947570823188, 1.83110845324441)  -- 
(1.38796125035268, 1.65442808484959)  -- 
(1.51471862575138, 1.43487787048422)  -- 
cycle ;
% Coordinates of gravity center of this polygon:
% (1.50405186144531, 1.64013813619274)

\filldraw[draw= black ,fill= yellow!60 ,opacity= 0.300000000000000 ]
(1.35924551806769, 0.784760765556715)  -- 
(1.43145750507289, 0.594345874280551)  -- 
(1.60947570823350, 0.478292623488118)  -- 
cycle ;
% Coordinates of gravity center of this polygon:
% (1.46672624379136, 0.619133087775128)

\filldraw[draw= black ,fill= yellow!60 ,opacity= 0.300000000000000 ]
(2.48528137419004, 1.43487787072917)  -- 
(2.61203874964732, 1.65442808484959)  -- 
(2.39052429176161, 1.83110845327163)  -- 
cycle ;
% Coordinates of gravity center of this polygon:
% (2.49594813853299, 1.64013813628346)

%% Noms des regions:
\draw (2,1.15) node{$K$};
\draw (1.6, 0.86) node{{\tiny $s_{\alpha}\!\cdot\!K$}};

\end{tikzpicture}}
\end{minipage}%
\caption{A {\em compact hyperbolic root system.} Picture in rank $3$ of the first normalized roots converging toward the set of limits of roots $E(\Phi)$ for the root system with diagram in the upper left of the picture. In red is $\widehat Q=\partial \mathbb{H}^p_n$. Reflection hyperplanes for $W$ and $K$ are also represented. Note that here $K\cap \mathbb{H}^2_p=-C\cap \mathbb{H}^2_p$ is compact in $\mathbb{H}_p^2$.}
\label{fig:5}
\end{figure}

\end{itemize}
The example in Figure~\ref{fig:3} is neither compact hyperbolic nor hyperbolic, but only weakly hyperbolic, and $K\cap \mathbb{H}^2_p\subsetneq -C\cap \mathbb{H}^2_p$.

In the specific case where $\Delta$ is a basis of the Lorentzian space (i.e., the rank of $W$ as a Coxeter group is $n+1$), these notions correspond to the usual definitions using the fundamental domain $P$ (defined in Proposition~\ref{prop:Tits}), see for example~\cite[\S6.8]{Hu90}:
\begin{itemize}
\item $W$ is hyperbolic if $P$ has finite volume;
\item $W$ is compact hyperbolic if $P$ is compact.
\end{itemize}

\begin{remark}\label{rem:hyp} We should be careful with the terminology here, because the properties studied often depend on the geometry, i.e., on $W$ as a discrete reflection group on $\mathbb{H}^n$ and not  only as an abstract Coxeter group. For instance, the  Coxeter systems corresponding to the based root systems in Figure~\ref{fig:3} and Figure~\ref{fig:4} are isomorphic as Coxeter systems, but as discrete reflection groups, one is hyperbolic and the other one is only weakly hyperbolic. 
\begin{enumerate}
\item For this reason, we do not say that a Coxeter group associated to a weakly hyperbolic root system is a ``weakly hyperbolic Coxeter group'': for example,  the universal Coxeter group of rank $4$ admits a representation with a weakly hyperbolic root system, but it also  admits a representation with a root system of signature $(2,2)$, which is therefore not weakly hyperbolic, see \cite[Remark 4.3 and Figure~5]{DyHoRi13}. 

\item In this setting, we should also avoid the terminology ``hyperbolic Coxeter group''. Indeed, for instance, the universal Coxeter group of rank $3$ admits a geometric representation attached to a weakly hyperbolic, non hyperbolic,  root system (see Figure~\ref{fig:3}), and a geometric representation attached to a hyperbolic root system (see Figure~\ref{fig:4}). This can happen when at least one of the labels of the Coxeter graph is $\infty$, i.e., when there are several choices for the value of $B(\alpha,\beta)\leq -1$ for some simple roots $\alpha,\beta\in \Delta$. 

\item In the case of compact hyperbolic root systems of rank $\geq 3$, if $\Delta$ is a basis of the Lorentzian space, then the hyperbolic hyperplanes associated to simple roots cannot be parallel or ultra-parallel (they always intersect). Indeed, these types of Coxeter groups are classified (they are such that all strict parabolic subgroups have finite type, see \cite[\S6.9]{Hu90}), and none of them admits a label $\infty$ in the Coxeter graph. This implies that the issues (1)-(2) do not arise in these cases.
  
  \item In the general case (when $\Delta$ is not necessarily a basis of the Lorentzian space), the root system may be compact hyperbolic even if there are some labels $\infty$ in the Coxeter graph (so, even if the associated (abstract) Coxeter group is \emph{not} compact hyperbolic in the sense of \cite[\S6.8]{Hu90}). For example, the reflection group on $\mathbb{H}^2$ generated by the reflections in the sides of a right-angled pentagon is compact hyperbolic.
\end{enumerate}
\end{remark}

   %% Hyperbolic Coxeter groups, and compact hyperbolic Coxeter groups, have been classified\footnote{A far as we know, Coxeter groups associated to weakly hyperbolic root systems have not been classified.}, see \cite[\S6.9]{Hu90}. 
   From \cite[Theorem 4.4]{DyHoRi13} we have the following characterization of the hyperbolic root systems by their limit sets.

\begin{prop} Let $(\Phi,\Delta)$ be an irreducible indefinite\footnote{This means that the Coxeter graph is connected and that the Coxeter group is neither finite nor affine, see for instance \cite{DyHoRi13} for more details.} based root system of rank $\geq 4$. Then $(\Phi,\Delta)$ is a hyperbolic root system if and only if $E(\Phi)=\widehat Q$.
\end{prop}

\begin{remark} The statement in \cite[Theorem 4.4]{DyHoRi13}  does not require the signature to be Lorentzian, it only requires that $(W,S)$ is irreducible. 
\end{remark}

In other words, with Theorem~\ref{thm:main} we obtain the following corollary.

\begin{cor} Let $W$ be a discrete reflection group on $\mathbb{H}^n$. Then $W$ is hyperbolic if and only if its limit set $\Lambda(W)$  is equal to $\widehat Q$.
\end{cor}

%%%%%%%%%%%%%%%%%%% PART 4%%%%%%%%%%%%%%%%%%%%%%%%%%%%%%%%

\section{An Example: Universal Coxeter Group and Apollonian Gaskets}\label{ss:univers}

As an illustration, we describe here, in the light of the preceding sections, the limit set  appearing in Figure \ref{fig:intro}, which is an Apollonian gasket. 

\subsection{Conformal models of the hyperbolic space}
We first introduce two conformal models of the hyperbolic space, the \emph{conformal ball model} and  the \emph{upper halfspace model} that turn out to be more practical to deal with the geometry because their  isometries are M\" obius transformations. For details we refer the reader to \cite[Chapter 4]{Ra06} and \cite[Chapter A]{BP92}. We use the notation $\|.\|$ for the Euclidean norm of $\mathbb{R}^n$.

\subsubsection{Inversions and the M\" obius group}\label{sec:mobiusgroup}

In the Euclidean space $\mathbb{R}^n$ endowed with its standard scalar product, let $\mathcal{S}(a,r)$ denote a sphere with center $a$ and radius~$r$. The \emph{inversion} with respect to $\mathcal{S}(a,r)$ is the map:
\[
\begin{array}{rcll}
i_{a,r}:&\mathbb{R}^n\setminus\{a\}&\longrightarrow &\mathbb{R}^n\setminus\{a\}\\
&x&\longmapsto &a+r^2.\frac{x-a}{\|x-a\|^2}
\end{array}
\] 
It is an involutive diffeomorphism that is  conformal and changes spheres into spheres. It extends to an involution $I_{a,r}$ of the one point compactification $\overline{\mathbb{R}^n}=\mathbb{R}^n\cup\{\infty\}$ by setting $I_{a,r}(a)=\infty$ and $I_{a,r}(\infty)=a$, which is a conformal involutive diffeomorphism once  $\overline{\mathbb{R}^n}$ is given its standard diffeomorphic and conformal structures. Then $I_{a,r}$ changes spheres/hyperplanes into spheres/hyperplanes. Its  set of fixed points is the whole sphere $\mathcal{S}(a,r)$, and conformality implies that a sphere/hyperplane $H$ is stable under $I_{a,r}$ if and only if $H$ intersects $\mathcal{S}(a,r)$ orthogonally.
Whenever $I_{a,r}$ changes the sphere  $\mathcal{S}_{b,\rho}$ into a hyperplane $H$ then $I_{a,r}\circ I_{b,\rho}\circ I_{a,r}^{-1}$ is the Euclidean (orthogonal) reflection with respect to $H$. 
The \emph{M\" obius group} of $\overline{\mathbb{R}^n}$ (or $\mathbb{R}^n$) is defined as the group generated by all inversions  in $\overline{\mathbb{R}^n}$ (note that, in particular, it contains all reflections in $\overline{\mathbb{R}^n}$).

\subsubsection{The Conformal Ball Model $\mathbb{H}_c^n$}

Consider the open unit ball embedded in the hyperplane $\mathbb{R}^n\times\{0\}$ of $V$:
\[
D^n=\{(x_1,\ldots,x_{n+1})\in V\,|\, x_{n+1}=0\ \text{and}\ x_1^2+\cdots+x_n^2<1\}
\]
and the 
\emph{stereographic projection} $c$ with respect to $-e_{n+1}$ of $\mathbb{R}^n\times\mathbb{R}_+^*$ onto $\mathbb{R}^n\times\{0\}$:
\[
\begin{array}{rcll}c:& \mathbb{R}^n\times\mathbb{R}_+^* &\longrightarrow &\mathbb{R}^n\times\{0\}\\
&(x_1,\ldots,x_{n+1})&\longmapsto &\frac{(x_1,\ldots,x_n,0)}{1+x_{n+1}}
\end{array}
\]
One verifies that $c$ restricted to the hyperboloid model $\mathbb{H}^n$ is a diffeomorphism onto $D^n$ (cf. \cite{Ra06, BP92}). Once $D^n$ is endowed with the pull-back metric
(the Riemannian metric $\mathrm{d}s=\frac{2\mathrm{d}x}{1-\|x\|^2}$) one obtains the \emph{conformal ball model} of the hyperbolic space, that we denote by $\mathbb{H}^n_c$ (cf. Figure \ref{fig:poincaremodel}).

The (hyperbolic) hyperplanes in $\mathbb{H}_c^n$ are the intersections with $D^n$ of the Euclidean spheres and hyperplanes in $\mathbb{R}^n\times\{0\}$ that are perpendicular to the boundary sphere $\partial\mathbb{H}^n_c:=\partial\overline{D^n}$. 
The hyperbolic reflection across the hyperplane $H$ is the restriction of the inversion with respect to the Euclidean sphere or hyperplane in $\mathbb{R}^n\times\{0\}$ containing $H$.  
\begin{figure}[!h]
\begin{minipage}[b]{\linewidth}
\centerline{
\includegraphics[scale=1]{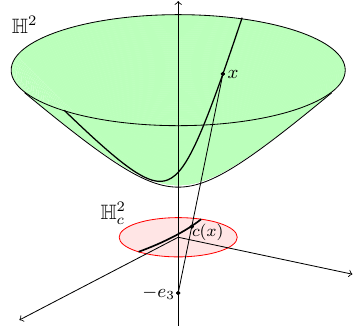}
}
\end{minipage}%
\caption{The conformal disk model.}
\label{fig:poincaremodel}
\end{figure}
It turns out that the group of isometries $\mathcal{I}(\mathbb{H}^n_c)$ is the subgroup of the M\" obius group of $\mathbb{R}^n\times\{0\}$ that leaves $D^n$ invariant  or, equivalently, generated by inversions with respect to spheres that are perpendicular to the boundary\footnote{It contains all reflections with respect to hyperplanes that are perpendicular to the boundary.}.  The model is conformal: the hyperbolic and Euclidean angles are the same.\\

The map $c\circ p:\mathbb{H}^n_p\longrightarrow\mathbb{H}^n_c$ is an isometry from the projective model to the conformal ball model and a simple computation shows that:
\[
c\circ p(x_1,\ldots,x_n,1)=\frac{(x_1,\ldots,x_n,0)}{1+\sqrt{1-x_1^2-\cdots-x_n^2}}.
\]
so that it obviously extends to a homeomorphism from $\overline{\mathbb{H}^n_p}=\mathbb{H}^n_p\cup\partial\mathbb{H}^n_p$ to $\overline{\mathbb{H}^n_c}=\mathbb{H}^n_c\cup\partial\mathbb{H}^n_c$ that restricted  to $\partial\mathbb{H}^n_p\longrightarrow\partial\mathbb{H}^n_c$ is the translation with vector $-e_{n+1}$.

\subsubsection{The Upper Half Space Model $\mathbb{H}_u^n$}\label{sec:upperhalfspace}  
Consider the differentiable map:
\[
u: \begin{array}{cll} D^n&\longrightarrow &\mathbb{R}^n\\
x&\longmapsto &2\,\dfrac{x+e_n}{\|x+e_n\|^2}-e_n
\end{array}
\]
 One verifies  (cf. \cite{BP92}, Chapter A) that $u$ is a diffeomorphism from $D^n$ onto the open upper half-space: $\mathbb{R}^{n-1}\times\mathbb{R}_+^*=\{x\in\mathbb{R}^n\,|\, x_n>0\}$, which, in fact, is the inversion with respect to the sphere with radius $\sqrt{2}$ and center $-e_n$ (cf. \S\ref{sec:mobiusgroup} and Figure \ref{fig:upperhalfspace}). Once $D^n$ is identified with the conformal ball model $\mathbb{H}^n_c$ and $\mathbb{R}^{n-1}\times\mathbb{R}_+^*$ is endowed with the pull-back metric with respect to $u^{-1}$, we obtain the \emph{upper half-space model} $\mathbb{H}_u^n$ of the hyperbolic space with Riemannian metric $\mathrm{d}s^2=(\mathrm{d}x_1^2+ \dots + \mathrm{d}x_{n-1}^2+{\mathrm{d}x_n^2})/{x_n^2}$. The hyperplanes in $\mathbb{H}^n_u$ are euclidean half-spheres with centers on the boundary $\mathbb{R}^{n-1}\times\{0\}$ as well as vertical affine hyperplanes. The model is conformal: hyperbolic angles agree with Euclidean ones. A reflection with respect to a hyperplane $H$ is a Euclidean reflection with respect to $H$ (when $H$ is a ``vertical'' Euclidean hyperplane) or an inversion with respect to $H$ (when $H$ is a ``half-sphere''). 
\begin{figure}[!h]
\begin{minipage}[b]{\linewidth}
\centerline{
\includegraphics[scale=0.5]{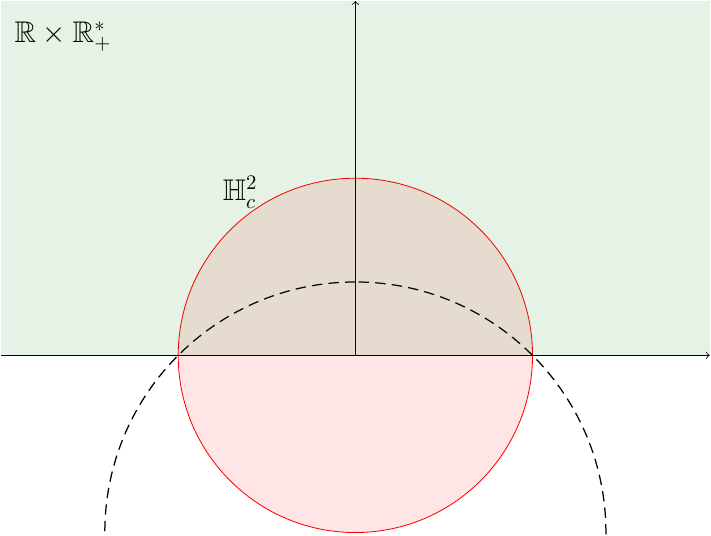}\qquad
\includegraphics[scale=0.5]{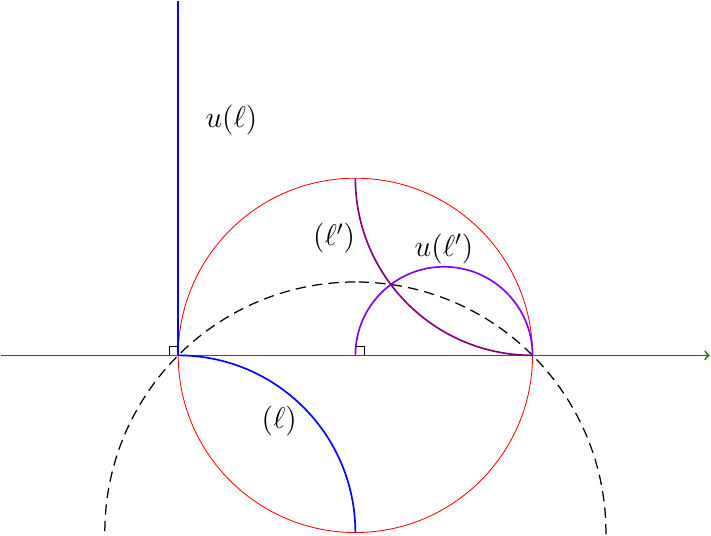}
}
\end{minipage}%
\caption{The inversion $u$ with respect to the sphere with center $-e_n$ and radius $\sqrt{2}$ (in dash) sends the conformal disk model (in red) onto the upper half-plane model (in green). On the right side two infinite geodesics in $\mathbb{H}^2_c$, $(\ell)$ and $(\ell')$, and their images by $u$, which are geodesics of $\mathbb{H}^2_u$.}
\label{fig:upperhalfspace}
\end{figure}

The group of isometries of $\mathbb{H}^n_u$ is the subgroup of the M\" obius group of $\mathbb{R}^n$ that stabilizes $\mathbb{R}^{n-1}\times\mathbb{R}_+^*$, or equivalently, the group generated by inversions\footnote{It contains all reflections with respect to hyperplanes perpendicular to $\mathbb{R}^{n-1}\times\{0\}$.} with respect to spheres perpendicular to the boundary $\mathbb{R}^{n-1}\times\{0\}$.   

The hyperbolic boundary $u(\partial\mathbb{H}^n_c)$ of $\mathbb{H}^n_u$ is the one point compactification $\partial\mathbb{H}^n_u:=\left(\mathbb{R}^{n-1}\times\{0\}\right)\cup\{\infty\}$.  \\

Both conformal ball models and upper half-space models are conformal, {\it i.e.}, hyperbolic angles (defined by their respective Riemannian metrics) coincide with Euclidean angles (with respect to the metrics induced by the natural embeddings in the Euclidean space). Furthermore, the hyperbolic circles in those models coincide with Euclidean circles, but usually with different center and radius (cf. \cite{Ra06, BP92}).

\subsection{Representation of the universal Coxeter group of rank 3 as a discrete subgroup of isometries of $\mathbb H^2$ with finite covolume}
\label{part:Z2*Z2*Z2}
Consider  an ideal triangle $(abc)$ in  $\mathbb{H}_c^2$, with $a,b,c$ three distinct points in $\partial \mathbb{H}_c^2$. Its sides $(ab)$, $(ac)$ and $(bc)$ are infinite geodesics that are pairwise parallel, therefore with angles 0, and $(abc)$ has a finite area equal to $\pi$ (\cite{Ra06}, \S 3.5, Lemma 4). Let $G\subset \mathcal{I}({\mathbb H^2_c})$
be the hyperbolic reflection group  generated by the (hyperbolic) reflections $s_{(ab)}$, $s_{(ac)}$, $s_{(bc)}$ with respect to the sides of $(abc)$. The group $G$ is a generalized simplex reflection group in the sense of \cite{Ra06} (cf. \S 7.3). The interior~$\mathcal{P}$ of the domain delimited by $(abc)$ is a fundamental region (cf. \cite{Ra06}, \S 6) for the action of $G$ on $\mathbb H^2_c$ so that $G$ is a discrete group (Theorem 6.6.3, \cite{Ra06}), and~$G$ together with its generating set  
is the universal Coxeter group of rank 3:
\[ G= \left< s_{(ab)}, s_{(ac)}, s_{(bc)} \middle| s_{(ab)}^2=s_{(ac)}^2=s_{(bc)}^2=1 \right>\]
 isomorphic to $\mathbb{Z}/2\mathbb{Z}*\mathbb{Z}/2\mathbb{Z}*\mathbb{Z}/2\mathbb{Z}$ (Theorem 7.1.4, \cite{Ra06}); see also Figure~\ref{fig:4}.
\begin{figure}[!h]
\centerline{\includegraphics[scale=0.7]{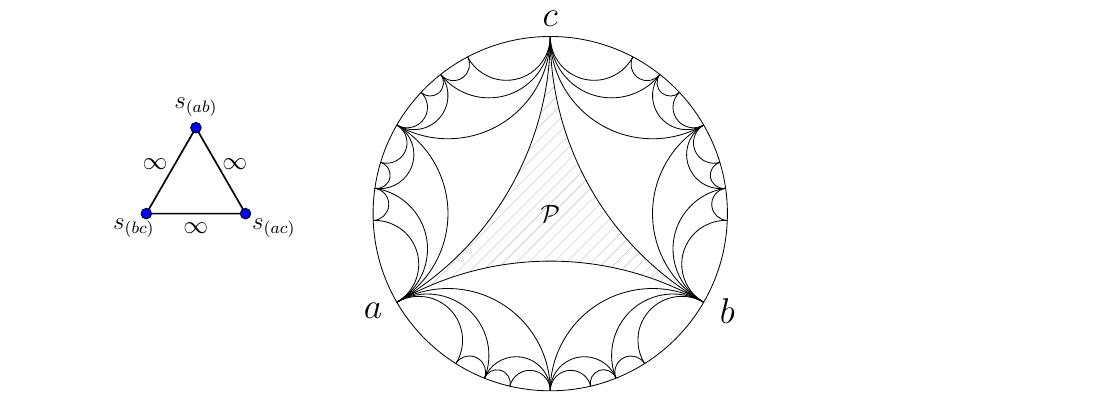}}
\caption{The fundamental domain $\mathcal{P}$ and some of its $G$-translates for the action of $G$ onto the conformal disk model $\mathbb{H}^2_c$.}
\label{jppfig1}
\end{figure}

The limit set $\Lambda(G)=\overline{G\cdot x}\smallsetminus G\cdot x$ of $G$ lies in $\partial \mathbb{H}^2_c$ (Theorem 12.1.2, \cite{Ra06}); it turns out that $\Lambda(G)$ is the whole of $\partial \mathbb{H}^2_c$:

\begin{lem}\label{limitsetG} The limit set of $G$ is
$\Lambda(G)=\partial\mathbb{H}^2_c$.
\end{lem}

This property can be deduced from the fact that this geometrical representation of the universal Coxeter group~$G$ of rank~$3$ corresponds to a hyperbolic Coxeter root system (see \S3.5.3). Consequently, the fact that the limit set of~$G$ is the whole boundary $\partial\mathbb{H}^2_c$ follows from a more general result stated in~\cite[Theorem 4.4]{DyHoRi13}. We give here a direct proof for convenience.

\begin{proof} Let $\omega\in\partial\mathbb{H}^2_c$;  it suffices to prove that there exists $x\in \mathbb{H}^2_c$ and a sequence $(u_n)_n$ in $G\cdot x$ that converges to $\omega$. Let $x$ be an arbitrary point in the fundamental domain $\mathcal{P}$. We suppose without loss of generality that $\omega\not=a$ and consider the upper half-space model $\mathbb{H}^2_u$  with $a$ sent to $\infty$. In this model the geodesics $(ab)$ and $(ac)$ become vertical lines that enclose a closed region $\mathcal{W}^0$ in $\mathbb{R}\times\mathbb{R}_+$, and $\mathbb{R}\times\mathbb{R}_+$ is the union of the countable family $(\mathcal{W}^n)_{n\in\mathbb{Z}}$ where $\mathcal{W}^n=(s_{(ac)}\circ s_{(ab)})^n(\mathcal{W}^0)$ are $G$-translates of $\mathcal{W}^0$.  

Let $\mathcal{D}^0={\mathcal{W}^0\setminus \mathcal{P}}\subset \mathcal{W}^0$; $\mathcal{D}^0$ is the half-disk domain in $\mathbb{R}\times\mathbb{R}^+$ delimited by $(bc)$ and it contains an uppermost  $G$-translate of $\mathcal{P}$: the ``triangular" domain $\mathcal{P}^0=s_{(bc)}(\mathcal{P})$  with side $(bc)$ (hatched in Figure \ref{jppfig2}). Let $x^0$ denote the unique point in $\mathcal{P}^0\cap G\cdot x$. The two other sides of $\mathcal{P}^0$, namely the infinite parallel hyperbolic geodesics $s_{(bc)}(ac)$ and $s_{(bc)}(ab)$ cut $\mathcal{D}^0$ into $\mathcal{P}^0$ and two half-disk domains $\mathcal{D}^0_0$ and $\mathcal{D}^0_1$. As above $\mathcal{D}^0_i$ ($i=0,1$) decomposes into a triangular domain $\mathcal{P}^0_{i}$, which is a $G$-translate of $\mathcal{P}$,  and two half-disk domains $\mathcal{D}^0_{i0}$ and $\mathcal{D}^0_{i1}$ and we set $x^0_{i}=\mathcal{P}^0_i\cap G\cdot x$, and so on: in this construction for each finite sequence $\sigma$ of $0$ and $1$,  $\mathcal{D}^0_\sigma$ is a half-disk domain in $\mathcal{D}^0$ that decomposes into a $G$-translate of $\mathcal{P}$ and two half-disk domains $\mathcal{D}^0_{\sigma 0}$ (the left one) and $\mathcal{D}^0_{\sigma 1}$ (the right one) and we set $x^0_\sigma$ the unique point in $\mathcal{P}^0_\sigma\cap G\cdot x$. A similar construction yields for all $n\in\mathbb{Z}$ and for any finite sequence $\sigma$ in $\{0,1\}$, a half-disk domain $\mathcal{D}^n_\sigma=(s_{(ac)}\circ s_{(ab)})^n(\mathcal{D}^0_\sigma)$ and a point $x^n_\sigma= (s_{(ac)}\circ s_{(ab)})^n(x^0_\sigma)$ that lies in $\mathcal{D}^n_\sigma\cap G\cdot x$ (see Figure \ref{jppfig2}).
\begin{figure}[!h]
\centerline{\includegraphics[scale=0.7]{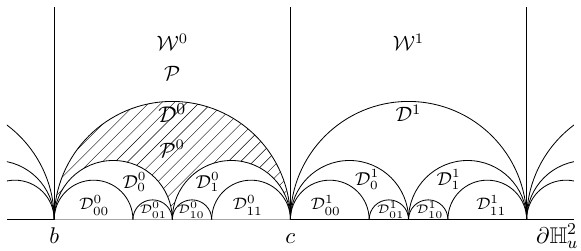}}
\caption{}
\label{jppfig2}

\end{figure}
%

%%% MODIFICATION ETE 2019 %%%%%%%
We apply an isometry of $\mathbb{H}^2_c$ that sends $b$ to 0, $c$ to 1 and $\infty$ to $\infty$, so that we suppose in the following that $b=0$ and $c=1$.
It yields the well-known Farey triangulation of $\mathbb{H}^2_c$ (see for instance \cite{Series}). It is intimately related to the Stern-Brocot tree, that enumerates all rational numbers between 0 and 1.

Let us recall some basic facts about the Stern-Brocot tree and highlight the link with the triangulation obtained.\\
Start from the two rational numbers $\frac01$ and $\frac11$ on the first row. The next row is obtained from the previous by inserting between two consecutive rationals $\frac{m}{n}$ and $\frac{m'}{n'}$ the rational $\frac{m+m'}{n+n'}$ ; repeat infinitely this process. The Stern-Brocot tree is the graph whose :
\begin{enumerate}
\item[--] vertices are the rational that appear on a row and do not already appear on the previous row ; $\frac01$ and $\frac11$ do not count as vertices.
\item[--] edges connect the vertex $\frac{m}{n}$ or $\frac{m'}{n'}$ on a row with the vertex $\frac{m+m'}{n+n'}$ on the next row (see left part of Figure \ref{sternBrocot}). 
\end{enumerate}
and it is easily checked that this process defines an infinite binary tree.

One can show by induction that two consecutive rationals $\frac{m}{n}$ and $\frac{m'}{n'}$ on a row and in this order, satisfy $m'n-mn'=1$ (it turns out that the converse is also true). In particular,  the sequences of rationals on a row are all increasing and by Bézout's theorem, all rationals obtained are irreducible. It's not difficult to state (and a well-known result) that all rational numbers (strictly) between 0 and 1 appear exactly once as a vertex of the Stern-Brocot tree. 

The link between the Stern-Brocot tree and the Farey triangulation follows from the following fact that can be obtained by a direct computation  (or see \cite{Series}): for any pair of rationals $\frac{m}{n}$ and $\frac{m'}{n'}$ with $|mn'-m'n|=1$, the point $\frac{m'}{n'}\in\partial\mathbb{H}^2_c$ is sent to $\frac{2m+m'}{2n+n'}\in\partial\mathbb{H}^2_c$ under the reflection across the geodesic with endpoints $\frac{m}{n}$ and $\frac{m+m'}{n+n'}$.

Since $\mathcal{D}^0$ has endpoints $\frac{0}{1}$ and $\frac{1}{0}$ and $\mathcal{D}^0_{0}$, $\mathcal{D}^0_{1}$ are parallel with common endpoint $\frac12$, it follows from the previous fact that the endpoints of $\mathcal{D}^0_{\sigma}$ for any finite sequence $\sigma$ in $\{0,1\}$, are consecutive rationals on a row of the Stern-Brocot tree construction. More precisely, the $\mathcal{D}^0_{\sigma}$ for any finite sequence $\sigma$ in $\{0, 1\}$ of length $p$ are in 1-1 correspondence with consecutive rationals on the $(p+1)$-st row of the Stern-Brocot tree  (see Figure \ref{sternBrocot}).   

\begin{figure}[ht]
\centerline{\includegraphics[scale=0.6]{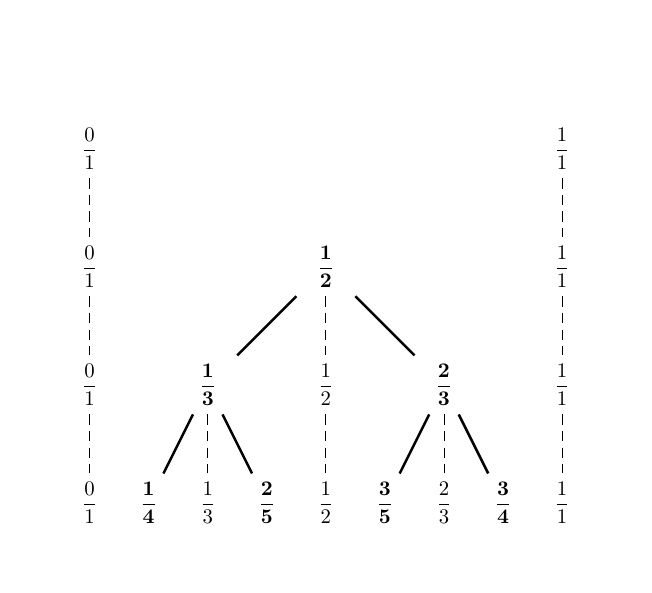}\quad \includegraphics[scale=0.6]{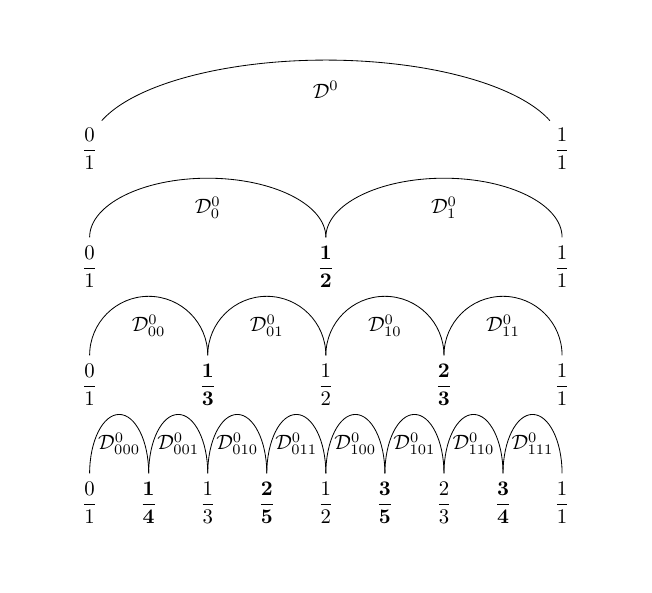}}
\caption{The four first lines of the Stern-Brocot tree (on the left)  and the endpoints of $\mathcal{D}^0_{\sigma}$ with $\sigma$ a sequence of 0,1 of length $<4$ (on the right).}
\label{sternBrocot}
\end{figure}

We now define by induction a sequence $(u_p)_{p\in\mathbb{N}}$ in $G\cdot x$ that converges to $\omega$. By construction $\omega$ lies in (at least) one domain $\mathcal{D}^k$, and up to a translation $(s_{(ab)}\circ s_{(ac)})^k$ we suppose that $\omega$ lies in $\mathcal{D}^0$ ; then set $u_0=x^0$. Whenever $\sigma$ is a (possibly empty) finite sequence of length $p\in\mathbb N$  in $\{0,1\}$ such that $\omega\in \mathcal{D}^0_\sigma$ and $u_p= x^0_\sigma$   then necessarily $\omega\in \mathcal{D}^0_{\sigma 0}$ or $\omega \in\mathcal{D}^0_{\sigma 1}$ (possibly both),   then set respectively $u_{p+1}=x^0_{\sigma 0}$ or  $u_{p+1}=x^0_{\sigma 1}$. 

By construction  $u_p$ lies in a half-disk domain $\mathcal{D}^0_\sigma$ with $\sigma$ a finite sequence in $\{0,1\}$ of length $p$. The two endpoints of $\mathcal{D}^0_\sigma$ are consecutive rationals on the $(p+1)$-st row of the Stern-Brocot tree. To conclude that $\lim u_p = \omega$, it suffices to prove that two consecutive rationals $\frac{m}{n}$ and $\frac{m'}{n'}$ on the $p$-th row of the Stern-Brocot tree satisfy: $\left|\frac{m}{n}-\frac{m'}{n'}\right| \le \frac{1}{p}$. We proceed by induction:
on the first row ($p=1$), the two consecutive rationals are $\frac01$ and $\frac11$ and the assumption holds.\\
Suppose the assumption is true on row $p$, and consider $\frac{m}{n}$ and $\frac{m+m'}{n+n'}$ two consecutive rationals on the $(p+1)$-st row. By hypothesis one has $\left|\frac{m}{n}-\frac{m'}{n'}\right|=\frac{1}{nn'}\leq \frac{1}{p}$, hence $nn'\ge p$. Therefore $\left|\frac{m}{n}-\frac{m+m'}{n+n'}\right|=\frac{1}{n(n+n')} = \frac{1}{n^2+nn'}\leq \frac{1}{p+1}$ since $nn'\ge p$ and $n^2\ge 1$. The assumption remains true on row $p+1$, which concludes the proof. (Note that we have proven, along the way, the density of $\mathbb{Q}$ in $\mathbb{R}$).
\end{proof}

\subsection{The Apollonian gasket}\label{sec:Apolloniangasket} We consider in the conformal disk  model $\mathbb{H}^2_c$ the three infinite geodesics  $(ab)$, $(ac)$ and $(bc)$, as above (see Figure \ref{jppfig1}).

\begin{lem}\label{jpplem1}  For any  three distinct points $a,b,c$ in $\partial\mathbb{H}^2_c$:
\begin{itemize}
\item[(i)] there exist unique  horocycles $h_a$, $h_b$, $h_c$ with limit points $a,b,c$ that are pairwise tangent.
\item[(ii)] $(ab)$ intersects $h_a$ and $h_b$ perpendiculary at the point $h_a\cap h_b$.
\item[(iii)] There exists a unique circle $\mathcal{C}$ passing through the intersection points $h_a\cap h_b$, $h_a\cap h_c$ and $h_b\cap h_c$; moreover $\mathcal{C}$ is tangent to the three geodesics $(ab)$, $(ac)$ and $(bc)$. 
\end{itemize}
\end{lem}

\begin{proof}
 In the upper half space model $\mathbb H_u^2$  with $c$ sent to $\infty$ (see Figure~\ref{lemme_upperhalf}) the horocycle $h_c$ becomes the horizontal line $y=2r$, $h_a$ and $h_b$  become circles  tangent  to the boundary $\mathbb R\times\{0\}$ respectively in $a$ and $b$; therefore the horocycles are pairwise tangent if and only if $h_a$, $h_b$ both have radius $r$ and $2r=\|a-b\|$. This proves (i).\\
 \begin{figure}[h]
\centerline{\includegraphics[scale=0.7]{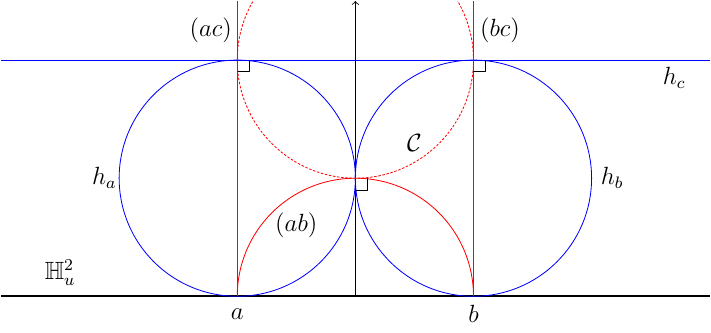}}
\caption{}
\label{lemme_upperhalf}
\end{figure}
  The geodesics $(ac)$ and $(bc)$ become vertical lines $x=a$ and $x=b$ while $(ab)$ is a half-circle with diameter the segment $[a,b]$. One obtains (ii). 
  
  The three tangency points are not aligned, hence there exists a unique circle $\mathcal{C}$ passing through them; it  has Euclidean diameter the segment with extremities $h_a\cap h_c$ and $h_b\cap h_c$, Euclidean radius $r$, and is perpendicular to the three horocycles; with (ii), $\mathcal{C}$ is tangent to the three geodesics. This proves (iii). 
\end{proof} 

\begin{figure}[!h]
\centerline{\includegraphics[scale=0.7]{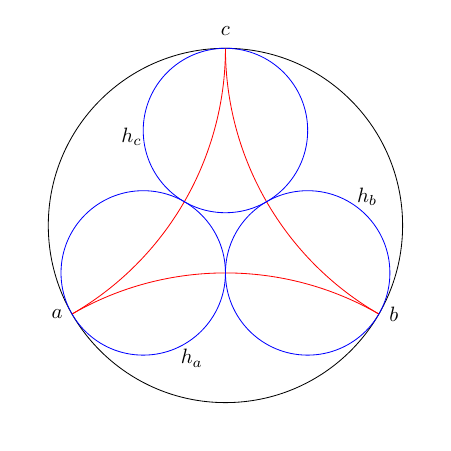}\quad\includegraphics[scale=0.7]{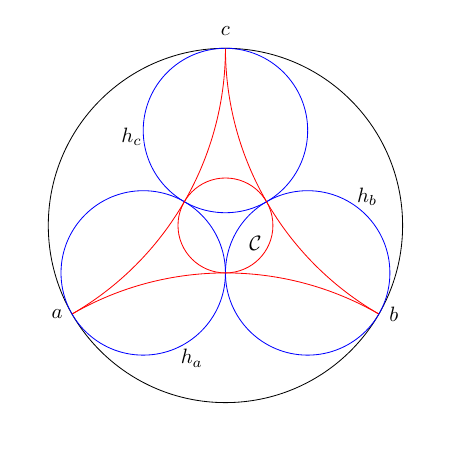}}
\caption{The geodesics $(ab)$, $(ac)$, $(ab)$, horocycles $h_a$, $h_b$, $h_c$ and the circle $\mathcal{C}$ in the conformal disk model.}
\label{mobius0}
\end{figure}

In the conformal disk model $\mathbb{H}_c^2$ the three geodesics $(ab)$, $(ac)$ and $(bc)$ lie in three circles, respectively $\mathcal{C}_{(ab)}$, $\mathcal{C}_{(ac)}$ and $\mathcal{C}_{(bc)}$ of the Euclidean plane. Let $\overline{G}$ be the subgroup of the M\"obius group of $\mathbb{R}^2$ generated by the inversions with respect to   $\mathcal{C}_{(ab)}$, $\mathcal{C}_{(ac)}$,  $\mathcal{C}_{(bc)}$, and the circle $\mathcal{C}$ given by Lemma \ref{jpplem1}, see Figure~\ref{mobius0}. The group $G$ of \S\,\ref{part:Z2*Z2*Z2} identifies with the subgroup generated by the inversions across $\mathcal{C}_{(ab)}$, $\mathcal{C}_{(ac)}$, and $\mathcal{C}_{(bc)}$ and accordingly $\overline{G}$ is generated by $s_{(ab)}$, $s_{(ac)}$, $s_{(bc)}$ and the inversion $s_{\mathcal{C}}$ with respect to $\mathcal{C}$.

Since the inversion with respect to $\mathcal{C}$ preserves a circle/line if and only if $\mathcal{C}$ intersects the circle/line at right angles (cf. \S\ref{sec:mobiusgroup}),  Lemma~\ref{jpplem1} implies that $s_{\mathcal{C}}$ preserves each of the horocycles $h_a$, $h_b$, $h_c$; each one of the hyperbolic reflections $s_{(ab)}$, $s_{(ac)}$ and $s_{(bc)}$ preserves the two horocycles  that intersect their axis (respectively $h_a,h_b$; $h_a, h_c$; and $h_b,h_c$) and moves the remaining one (respectively $h_c$, $h_b$ and $h_a$)  to a horocycle that remains tangent  to the two others. The orbit of  the three horocycles $h_a$, $h_b$, $h_c$ and of $\partial\mathbb{H}^n_c$  under the action of $\overline{G}$  yields a configuration of pairwise tangent or disjoint circles in $\overline{\mathbb{H}}^2_c$, see Figure \ref{mobius}. This configuration is called an \emph{Apollonian gasket} $\mathcal{A}$ and is widely studied in the literature, see for instance~\cite{ap2,ap5,ap1,ap4,ap3}.

 \begin{figure}[!h]
\centerline{\includegraphics[scale=0.4]{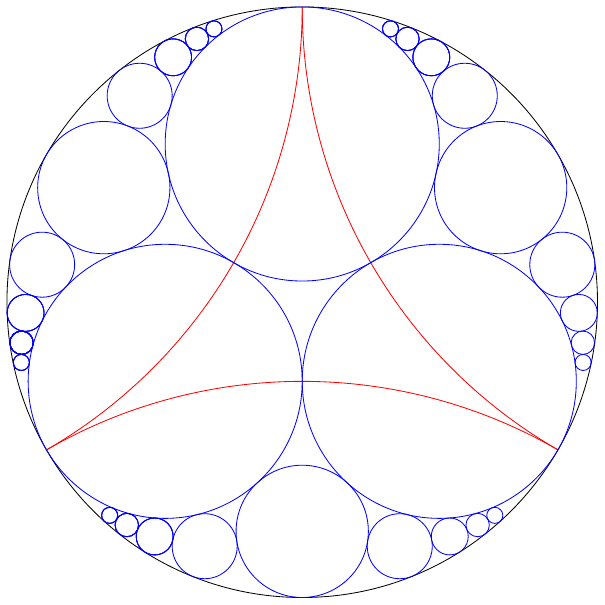}\qquad\qquad
\includegraphics[scale=0.4]{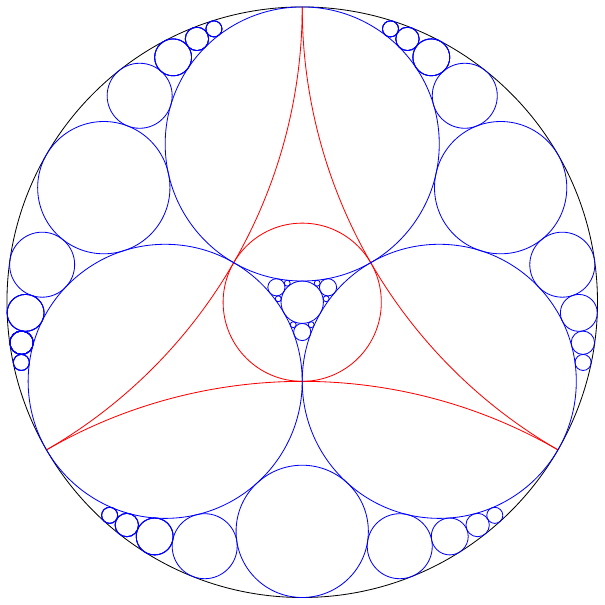}}
\caption{On the left (respectively right) some of
the orbit of the three horocycles (respectively and of the boundary circle) under the action of $G$ (respectively $\overline{G}$). The complete orbit on the right figure yields the \emph{Apollonian gasket}.}
\label{mobius}
\end{figure}

\subsection{Discrete representation in $\mathcal{I}(\mathbb{H}^3)$ of the universal Coxeter group with rank 4}

Consider the universal Coxeter group of rank 4 and its representation as a discrete subgroup of $O_B(V)$ with $(V,B)=\mathbb{R}^{3,1}$ and simple system $\Delta=\{\alpha,\beta,\gamma,\delta\}$ such that  for all distinct $\chi,\xi\in\Delta$, $B(\chi,\xi)=-1$.
In Figures~\ref{fig:intro} and~\ref{dim3fig1} are represented the polytope $\mathrm{conv}(\widehat{\Delta})$ and the unit ball $D_1^3=\widehat{Q^-}$ that we identify here with $\mathbb{H}^3_p$.
 \begin{figure}[!h]
\centerline{\includegraphics[scale=0.9]{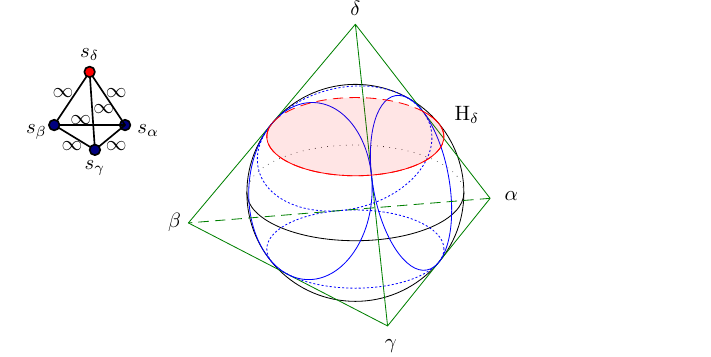}}
\caption{The unit ball $D_1^3$ with interior identified with the projective disk model $\mathbb{H}^3_p$; $\mathrm{conv}(\widehat{\Delta})$ is a regular tetrahedron such that the unit sphere $\partial D_1^3$ passes through the midpoints of its edges. In blue the circles on which $\partial D_1^3$ intersects $\mathrm{conv}(\widehat{\Delta})$. In red the  plane $H_\delta$ in $\mathbb{H}^3_p$ passing through 3 of these points.}
\label{dim3fig1}
\end{figure}

As discussed in \S\,3 that makes the Coxeter group $\Gamma$ act on the projective ball model by isometries (the action is given in Corollary~\ref{prop:projisom}), yielding a discrete faithful representation of $\Gamma$ in $\mathcal{I}(\mathbb{H}^3_p)$.

A direct computation shows that the reflection $s_\alpha$ acts as the hyperbolic reflection across the hyperbolic plane  $\mathrm{H}_\alpha=\alpha^\perp\cap \mathbb{H}_p^3$ passing through the midpoints of the three edges issued from the vertex $\alpha$ of the tetrahedron; indeed, for example, $B(\alpha,\frac{\alpha+\beta}{2})=\frac{1}{2}(B(\alpha,\alpha)+B(\alpha,\beta))=\frac{1}{2}(1-1)=0$ and the same computation shows that the midpoints of the three edges issued from $\alpha$ lie in $\alpha^\perp$. Consider also the reflection planes $H_\beta,\,H_\gamma,\,H_\delta$ (in red in Figure \ref{dim3fig1}) respectively associated to $s_\beta$, $s_\gamma$ and $s_\delta$, passing through the midpoints of the edges adjacent respectively to the vertices $\beta$, $\gamma$ and $\delta$. They are pairwise parallel and non ultra-parallel (they meet on the boundary $\partial\mathbb{H}_p^3$), which can also be seen by $\cosh\,d(H_\alpha,H_\beta)=|B(\alpha,\beta)|=1$.   

The boundary sphere $\partial \mathbb{H}_p^3$ intersects the faces of the tetrahedron along circles; for $\chi=\alpha,\beta,\gamma,\delta$ let us denote  by $h_\chi$ the intersection circle on the face of the tetrahedron opposite to the vertex $\chi$ (in blue in Figure \ref{dim3fig1}).

In the upper half-space model, the planes $H_\alpha$, $H_\beta$ $H_\gamma$ and $H_\delta$  yield a configuration of 4 half-spheres that are pairwise tangent, see Figure~\ref{dim3fig2}. The action of $\Gamma$ restricts on $\partial\mathbb{H}_u^3$  as the action of the subgroup $\overline{G}$ of the M\"obius group on the plane (see \S \ref{sec:Apolloniangasket}).

 \begin{figure}[!h]
\centerline{\includegraphics[scale=0.7]{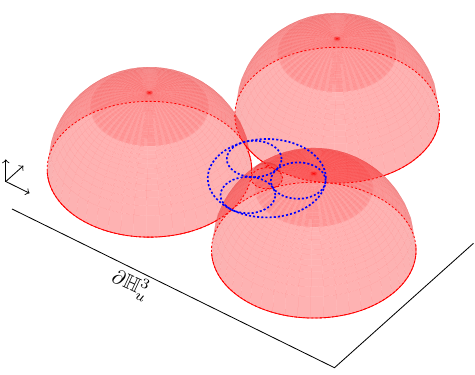}\qquad\quad
\includegraphics[scale=0.7]{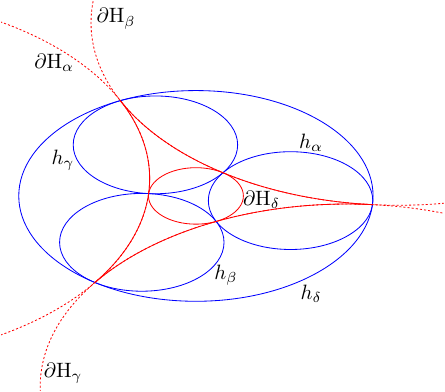}}
\caption{On the left picture: in red the four hyperbolic planes $H_\alpha$, $H_\beta$, $H_\gamma$ and $H_\delta$ (the small one in between) in the upper half-space model $\overline{\mathbb{H}^3_u}$, and in blue the four circles $h_\alpha$, $h_\beta$, $h_\gamma$ and $h_\delta$ in $\partial\mathbb{H}^3_u$. On the right picture their intersections with $\partial\mathbb{H}^3_u$.}
\label{dim3fig2}
\end{figure}

\begin{prop}
The limit set $\Lambda(\Gamma)$ of $\Gamma$ in $\overline{\mathbb{H}^3_u}$ is the closure in $\partial\mathbb{H}^3_u$ of the Apollonian gasket $\mathcal{A}$.
\end{prop}

\begin{proof} The two hyperplanes $H_\alpha$ and $H_\delta$ are parallel so denote by $x_0\in\partial\mathbb{H}^3_u$ their asymptotic point. Hence the composite of the two reflections with respect to $H_\alpha$ and $H_\delta$ is a transformation of parabolic type  with limit point $x_0\in\partial\mathbb{H}^3_u$ (cf. Propositions A.5.12 and A.5.14 of \cite{BP92}). According to Theorem 12.1.1 of \cite{Ra06}, $x_0\in\Lambda(\Gamma)$. Note that $x_0$ lies in the interior disk of $\mathbb{R}^2 \simeq \partial\mathbb{H}^3_u\setminus\{\infty\}$ delimited by the circle $h_\delta$.

 The action of $\Gamma$ on $\mathbb{H}^3_u$ naturally extends to a conformal action on $\overline{\mathbb{H}^3_u}$  that  is the Poincar\'e extension of the action of the subgroup $\overline{G}$ of the M\" obius group on  $\partial\mathbb{H}^3_u\simeq\overline{\mathbb{R}^2}$ (cf. \S\ref{sec:upperhalfspace} as well as Theorem 4.4.1 of \cite{Ra06}). As in \S \ref{sec:Apolloniangasket} denote by $G$ the subgroup of $\overline{G}$ generated by the three inversions with respect to the spheres $\partial H_\alpha$, $\partial H_\beta$ and $\partial H_\gamma$; $G$ acts on the interior of the disk delimited by $\partial H_\delta$ as the group $G$ of \S \ref{part:Z2*Z2*Z2} acts on $\mathbb{H}^2_c$. In particular $\Lambda(\Gamma)$ contains $\overline{G\cdot x_0}\setminus G\cdot x_0$ that equals $h_\delta$ (Proposition \ref{limitsetG}).
 
 After conjugating $\Gamma$ by the reflection with respect to $H_\alpha$ (respectively $H_\beta$, $H_\gamma$) the same argument applies to show that $\Lambda(\Gamma)$ contains also $h_\alpha$ (respectively $h_\beta$, $h_\gamma$). Hence the Apollonian gasket $\mathcal{A}$, as seen in \S \ref{sec:Apolloniangasket}, which is the orbit of $h_\alpha\cup h_b\cup h_\gamma\cup h_\delta$ under the action of $\overline{G}$, is a $\Gamma$-invariant subset of $\Lambda(\Gamma)$; since $\Lambda(\Gamma)$ is closed in $\partial\mathbb{H}^3_u$ (Theorem 12.1.2, Corollary 1 of \cite{Ra06}) the closure $\overline{\mathcal{A}}$ of $\mathcal{A}$ in $\partial\mathbb{H}^3_u$ is a closed $\Gamma$-invariant subset of $\partial\mathbb{H}^3_u$ contained in $\Lambda(\Gamma)$. Since $\Lambda(\Gamma)$ is infinite, $\Gamma$ is non elementary (cf. Theorem 12.2.1 of \cite{Ra06}) and therefore any closed $\Gamma$-invariant subset of $\partial \mathbb{H}^3_u$ contains $\Lambda(\Gamma)$ (Theorem 12.1.3 of \cite{Ra06}). Hence $\Lambda(\Gamma)$ equals $\overline{\mathcal{A}}$.
\end{proof}

 \begin{figure}[!h]
{\includegraphics[width=\textwidth]{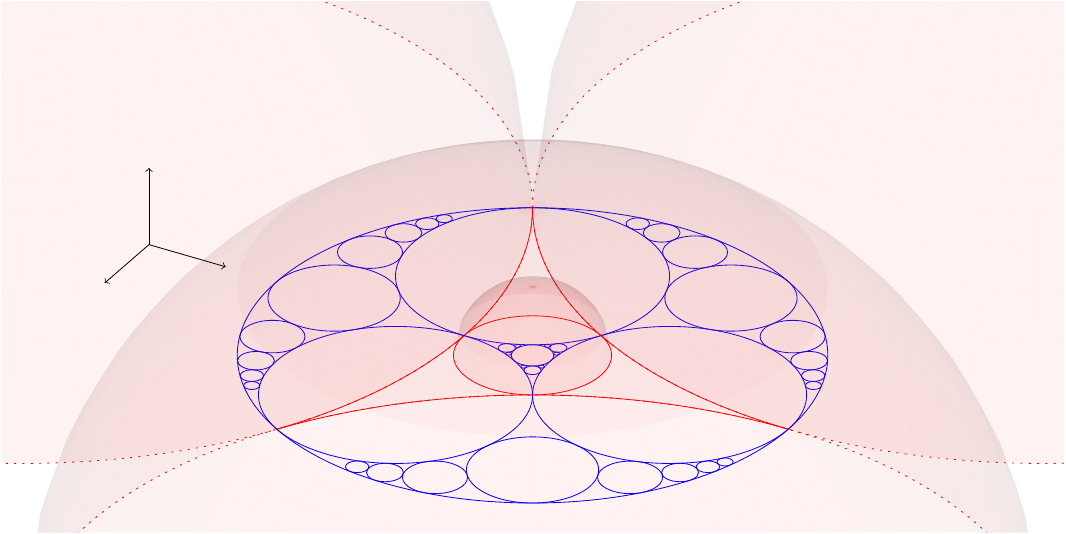}\qquad
}
\caption{The Apollonian gasket in $\partial\mathbb{H}^3_u$ whose closure is the limit set of $\Gamma$. In both the conformal and projective ball models one obtains as limit set the Apollonian packing of the sphere as in Figure \ref{fig:intro}.}
\label{dim3fig3}
\end{figure}

The closure $\overline{\mathcal{A}}$ of $\mathcal{A}$ is also the complement in the closed external disk (delimited by $h_\delta$) of the union of the interiors of disks delimited by the circles of the gasket (see~\cite[Theorem 4.10]{DyHoRi13} for an analogous property applying to the limit set of any discrete group generated by hyperbolic reflections).

\subsection*{Funding} During this work the first author was supported by a NSERC  grant and the third author was supported by a postdoctoral fellowship from  LaCIM. This collaboration was also made possible with the support of the UMI CNRS-CRM.

\subsection*{Acknowledgments} The authors wish to thank Jean-Philippe
Labb\'e who made the first version of the Sage and TikZ functions used to
compute and draw the normalized roots. The second author wishes to thank Pierre de la Harpe for his invitation to come to Geneva in June 2013 and for his comments on this article. The third author is grateful to Pierre Py for fruitful discussions in Strasbourg in November 2012. We also acknowledge the participation of  Nadia Lafreni\`ere and Jonathan Durand Burcombe to a LaCIM undergrad summer research award on this theme during the summer 2012.

The authors wish to warmly thank the anonymous referee for his/her helpful comments that improved the quality of this article. 

%%%%%%%%%%%%%%%%%%%% END OF PART 4 %%%%%%%%%%%%%%%%%%%%%%%%% %%%%%%%%%%%%%%%%% 

%%%%%%%%%%%%%BIBLIOGRAPHY%%%%%%%%%%%%%%%%%

\newcommand{\etalchar}[1]{$^{#1}$}

\end{document}